\documentclass[reqno,12 pt]{amsart}

\usepackage{amsfonts}
\usepackage{graphicx}
\usepackage{epstopdf}

\usepackage{amsfonts,amscd}
\usepackage{mathrsfs}
\usepackage{amssymb}
\usepackage{eufrak}
\usepackage{latexsym}
\usepackage{enumerate}
\usepackage{cite}
\usepackage{url}
\usepackage{amsmath}
\usepackage{verbatim}
\usepackage{color}
\usepackage{epsf}
\usepackage[width=6.8cm]{caption}
\usepackage{geometry}
\usepackage{hyperref}
\definecolor{dark-blue}{rgb}{0,0,0.6}
\definecolor{Purple}{rgb}{0.2,0,0.25}
\hypersetup{breaklinks=true,
pagecolor=white,
colorlinks=true,
citecolor=dark-blue,
filecolor=black,%
linkcolor=Purple,%
urlcolor=black
}
\newcommand{\bref}[1]{\textbf{\ref{#1}}} 
\newcommand{\beqref}[1]{\textbf{(\ref{#1})}} 

\usepackage[nobysame]{amsrefs}

\theoremstyle{plain} 
\newtheorem{thm}{Theorem}[section]
\newtheorem{lem}[thm]{Lemma}
\newtheorem{defin}[thm]{Definition}
\newtheorem{cor}[thm]{Corollary}
\newtheorem{prop}[thm]{Proposition}
\newtheorem{alg}[thm]{Algorithm}

\theoremstyle{definition}
\newtheorem{remark}[thm]{Remark}

\newtheorem{expl}[thm]{Example}
\newcommand{\wt}{\widetilde}

\newcommand{\dom}{\textnormal{dom}}

\newcommand{\R}{\mathbb{R}}

\newcommand{\N}{\mathbb{N}}

\newcommand{\Res}{\textnormal{Res}}

\newcommand{\cl}{\overline}


\begin{document}
\date{August 22, 2017}

\numberwithin{equation}{section}
\title[Inexact resolvent inclusion problems with applications]{Solutions to inexact resolvent inclusion problems  with applications to nonlinear analysis and optimization}
\author{Daniel Reem and Simeon Reich}
\address{Department of Mathematics, The Technion - Israel Institute of Technology, 3200003 Haifa, Israel}
\email{dream@technion.ac.il}
\address{Department of Mathematics, The Technion - Israel Institute of Technology, 3200003 Haifa, Israel}
\email{sreich@technion.ac.il}

\maketitle

\begin{abstract}
Many problems in nonlinear analysis and optimization, among them 
variational inequalities and minimization of convex functions, 
can be reduced to finding zeros (namely, roots) of set-valued 
operators. Hence numerous  algorithms have been devised in order 
to achieve this task. A lot of these algorithms are inexact 
in the sense that they allow perturbations to appear during the 
iterative process, and hence they enable one to better deal with 
noise and computational errors, as well as superiorization. 
For many years a certain fundamental question has remained open 
regarding many of these known inexact algorithmic schemes in 
various finite and infinite dimensional settings, namely whether 
there exist sequences satisfying these inexact schemes when 
errors appear. We provide a positive answer to this question. 
Our results also show that various theorems discussing the 
convergence of these inexact schemes have a genuine merit 
beyond the exact case. As a by-product we solve the standard 
and the strongly implicit inexact resolvent inclusion problems, 
introduce a promising class of functions (fully Legendre functions), 
establish continuous dependence (stability) properties of the 
solution of the inexact resolvent inclusion problem and continuity 
properties of the protoresolvent, and generalize the notion of 
strong monotonicity. 
\end{abstract}


\section{Introduction}\label{sec:Intro}
\subsection{Background}\label{subsec:Background} 
A central problem which appears in nonlinear analysis and optimization is the 
problem of finding zeros (namely, roots) of (usually nonlinear) operators. 
More precisely, given a finite or infinite-dimensional Banach space $(X,\|\cdot\|)$ 
with a dual $X^*$, and given a set-valued operator $A$ from $X$ to the set $2^{X^*}$ 
of all subsets of $X^*$, the problem of finding a zero of $A$ is the following problem:
\begin{equation}\label{eq:ZeroA}
\textnormal {to find }\,x\in X\,\,\,\textnormal{such that}\,\,\, 0\in Ax. 
\end{equation}
This problem is very useful because many other problems can be reduced to solving it (especially when the operator is maximally monotone), among them finding solutions of minmax problems, complementarity problems, variational inequalities, convex feasibility problems,  equilibrium problems, and minimizing convex functions. For  instance, given a proper, lower semicontinuous and convex function  $F:X\to (-\infty,\infty]$, if $A=\partial F$ (namely, the subdifferential of $F$), then, as is well known, $A$ is maximally monotone \cite[Theorem A]{Rockafellar1970jour} and $0\in Ax$ if and only if $x$ is a  global minimizer of $F$ (see \cite[Theorem 2.5.7, p. 105]{Zalinescu2002book}). As another example of the usefulness of \beqref{eq:ZeroA}, one observes that if $A$ is single-valued, then \beqref{eq:ZeroA} reduces to finding a solution $x\in X$ to the equation $0=Ax$.

As a result of the importance of \beqref{eq:ZeroA},  numerous (proximal-type) algorithms have been devised in order 
to achieve this task. Many of these algorithms are inexact in the sense that they allow perturbations (namely, error terms) to appear during the iterative process. The ability to allow perturbations has several advantages. First, in the implementation of various algorithmic schemes aiming at solving computational problems it is common that errors appear due to noise in the input (for instance, because of inaccurate measurements or noise in the transmission of the measurements), inaccurate computations (such as those resulting from subproblems involving proximal operators or other operators the exact evaluation of which is often too demanding a task), and so on. Inexact algorithms enable one to better deal with such error terms, especially when they are perturbation resilient, namely when they converge to a solution of the problem they aim to solve despite the presence of the perturbations. 
A second advantage of inexact algorithms aiming to solve \beqref{eq:ZeroA} can be found in the recent heuristic optimization methodology called ``superiorization''  \cite{Censor2015surv,CensorSuperiorizationPage,CensorDavidiHerman2010jour,Davidi2010PhD,Herman2014surv}. Here, in contrast to the previous case in which the perturbations themselves appear due to noise or computational errors and hence they are usually unknown to the users (frequently only their magnitude can be estimated), one uses perturbations in an active way in order to obtain partial solutions which have some good  properties. See \cite[Section 4]{ReemDe-Pierro2017jour} for a more comprehensive discussion regarding this methodology, including a significant extension of its scope and an extensive list of related references. 

It turned out that for many years (in some cases about 15--20 years), a certain fundamental question has remained open regarding many of the known inexact algorithmic schemes which were devised in order to solve \beqref{eq:ZeroA} in various finite 
and infinite dimensional settings. The question has been whether these algorithms are well defined. In other words, so far it has not been clear whether there exist any sequences satisfying these inexact schemes when errors appear. This question  is relevant not just to the inexact algorithmic schemes themselves, but also to many convergence results related to them, since on the one hand these convergence results assume that perturbations appear (and then impose some conditions on them, for instance, that they decay to zero), but, on the other hand, it has not been clear why the discussed algorithmic schemes are well defined in the presence of these non-zero perturbations. Hence so far there has been a real doubt regarding the merit of the above-mentioned  convergence results in the inexact case (a case which is expected to occur in real world scenarios).

\subsection{Contributions and paper layout}\label{subsec:Contributions}
After some preliminaries given in Section \bref{sec:Preliminaries}, we discuss in Section \bref{sec:FullyLegendre} the class of fully Legendre functions (and also fill small gaps in the literature). Then we discuss in Section \bref{sec:PrincipleResolvent} the inexact resolvent inclusion problem, a problem which is intimately related to many of the inexact algorithms mentioned above aiming to solve \beqref{eq:ZeroA}. We show that this problem has a unique solution in a rather general setting and, as a matter of fact, we are able to represent this solution in an explicit way (\beqref{eq:y_explicit} below).  This existence and uniqueness result, which is not very complicated, turned out to be very useful in our context. A major use of it (together with other tools such as the ones presented in Sections \bref{sec:Implicit}--\bref{sec:Continuity}: see the next paragraphs for more details) is answering, in a positive way, the many-years-open question mentioned above regarding the well-definedness of numerous (at least 16) known inexact algorithmic schemes in various finite and infinite dimensional settings, namely whether there exist sequences which satisfy these schemes when (some of) the error terms are not equal to zero. This is done in Sections \bref{sec:Eckstein}--\bref{sec:ManyMore} below (as discussed in Section \bref{sec:ManyMore} below,  probably our ideas and results can be applied to additional 14 or more algorithmic schemes). We show  that in some cases arbitrary perturbations are possible and in other cases (in which the algorithmic schemes are defined in a strongly implicit way) sufficiently small  perturbations are allowed. Sometimes we are also able to show that the domain of definition of additional  parameters which appear in some schemes can be extended (Sections \bref{sec:SolodovSvaiter1999-1}--\bref{sec:ParenteLotitoSolodov2008} below).

The results presented in this paper not only show the well-definedness of many existing algorithmic schemes (among them  the ones introduced by Eckstein \cite[Algorithm (10)]{Eckstein1998jour}, Reich-Sabach \cite[Algorithm (4.1)]{ReichSabach2010b-jour},  Solodov-Svaiter \cite[Algorithm 1]{SolodovSvaiter1999-1jour}, Iusem-Pennanen-Svaiter \cite[Method 1, Theorem 3]{IusemPennanenSvaiter2003jour}, and Parente-Lotito-Solodov \cite[Algorithm 3.1]{ParenteLotitoSolodov2008jour}), but they also show that numerous known theorems discussing the convergence of these inexact schemes (under the  assumption of existence of sequences satisfying the schemes) have a genuine merit beyond the exact case. 

Our analysis yields a few byproducts of independent interest. First, we introduce and investigate in Section \bref{sec:FullyLegendre} (see also Remarks \bref{rem:Legendre}--\bref{lem:FormulaConjugate}) the class of fully Legendre functions. This rich class of functions seems to be quite promising. Second, we show in Section \bref{sec:Continuity} that under simple assumptions (in particular,  when the space is finite dimensional: see Example \bref{ex:FiniteDimension}) there is continuous dependence (stability) of the solution of the inexact resolvent inclusion problem on other parameters which appear in the problem. In addition, we show the continuity of the protoresolvent. As a matter of fact, frequently these conditions ensure the H\"older continuity of the protoresolvent (Corollary \bref{cor:ResolventContinuous}, Examples \bref{ex:HilbertContinuous}-\bref{ex:PowerNorm} below), a property which generalizes the well-known 1-Lipschitz continuity (nonexpansivity) of classical resolvents in Hilbert space \cite[Corollary 23.9, p. 396]{BauschkeCombettes2017book}, \cite[Proposition 5.b]{Moreau1965jour}, \cite[Proposition 1(c)]{Rockafellar1976jour}. Third, we present in Section \bref{sec:Implicit} a strongly implicit form of the inexact  resolvent inclusion problem (Proposition \bref{prop:ImplicitInexactness} below). This form of the problem, together with the explicit representation \beqref{eq:y_explicit} of the solution  to the (standard) inexact resolvent inclusion problem and the continuity results mentioned a few lines above, are useful not only for showing that various inexact algorithmic schemes are well defined (such as the ones discussed in Sections \bref{sec:SolodovSvaiter1999-1}--\bref{sec:ParenteLotitoSolodov2008} and many ones  discussed in Section \bref{sec:ManyMore}), but also for devising many more strongly implicit ones. Fourth, in Definition \bref{def:ExtendStronglyMonotone} below we introduce a certain generalization of the notion of strong monotonicity. We conclude the paper in Section \bref{sec:AdditionalRemarks} with a few remarks and open problems. 

The  notation used in this paper is sometimes different from the one used in some of the cited references  because we wanted to have a consistent notation throughout the paper. Nonetheless, the differences are minor and should not lead to any confusion.

\section{Preliminaries}\label{sec:Preliminaries}
We first recall a few basic definitions. In order to make the discussion focused, we will consider the setting of Proposition \bref{prop:PrincipleResolvent} below, although some of the notions and definitions below can be easily generalized to a more general setting, say to functions from a topological vector space to  $(-\infty,\infty]$. Throughout the paper, unless stated otherwise, the notation and assumptions mentioned below will be used. 

Let  $(X,\|\cdot\|)$ be a real finite or infinite dimensional reflexive Banach space and let $(X^*,\|\cdot\|_*)$ be its dual. Let $2^{X^*}$ be the set of all subsets of $X^*$ and let $A:X\to 2^{X^*}$. We regard $A$ as being a set-valued (or ``multivalued'') operator from $X$ to $X^*$, that is, $A(x)$ is a subset of $X^*$ for each $x\in X$. We sometimes use the notation $Ax$ instead of $A(x)$. The effective  domain of $A$ is the set $\dom(A):=\{x\in X:  Ax\neq\emptyset\}$. The range of $A$ is the set $\cup_{x\in X}Ax$. We are interested only in  nontrivial operators $A$, that is, $\dom(A)\neq\emptyset$ (equivalently, the range of $A$ is nonempty). The set-valued  operator $A$ is called monotone if it satisfies the set-valued  monotonicity condition, that is, 
\begin{equation}\label{eq:Monotone}
\langle x_1^*-x_2^*,x_1-x_2\rangle \geq 0,\quad \forall x_1,x_2\in X,\, x_1^*\in Ax_1,\, x_2^*\in Ax_2,
\end{equation}
where $\langle x^*,x\rangle:=x^*(x)$ for all $x\in X$ and $x^*\in X^*$. 
We say that $A$ is maximally monotone  (the term ``maximal monotone'' is also frequently used in the literature) if $A$ satisfies the maximality condition with respect to monotonicity, that is, $A$ is monotone and for every other multivalued  monotone  operator $B$ from $X$ to $X^*$, if  $Ax\subseteq Bx$ for each $x\in X$, then $B=A$. In other words, if $A$ is maximally monotone and its graph $\{(x,x^*): x\in X,\, x^*\in Ax\}$ is contained in the graph of another monotone operator $B$, then $A=B$ (in particular, a maximally monotone operator cannot be trivial since the graph of the trivial operator is contained in the graph of any constant operator $B$, namely $Bx=x^*_0$ for each $x\in X$ where $x^*_0\in X^*$ is fixed). Well-known examples of maximally monotone  operators are subdifferentials of proper lower semicontinuous convex functions defined on a Banach space, the normal cone operator of a closed and convex subset of a Banach space, and continuous positive semi-definite (single-valued) linear operators from a Hilbert space to itself. Many examples, properties and applications of maximally monotone operators can be found in \cite{AuslenderTeboulle2003book, BauschkeCombettes2017book, Borwein2006jour, Brezis1973book, Phelps1993book_prep, Simons2008book}.

For each  $\lambda\in \R$, we denote by $\lambda A$ the set-valued operator from $X$ to $X^*$ defined by  $(\lambda A)(x):=\lambda A(x)$ for each $x\in X$. It is straightforward to  check that if $\lambda>0$, then $\lambda A$ is monotone whenever $A$ is monotone, and $\lambda A$ is maximally monotone whenever $A$ is maximally monotone. The zero set of $A$ is the set  $A^{-1}(0)=\{z\in X: 0\in Az\}$. We say that $A$ is single-valued if for each $x\in X$ the subset  $A(x)$ is nonempty and contains exactly one element from $X^*$; in other words,  $A$ can be regarded as an ordinary function from $X$ to $X^*$ and by abuse of notation we will identify the set $A(x)$ with the unique element that it contains. We denote by $I$ the identity operator, namely the single-valued operator $I:X\to X$ defined by $I(x):=x$ for each $x\in X$. 

The convex conjugate (Fenchel conjugation, Legendre-Fenchel transform, Legendre transform) of a function  $f:X\to(-\infty,\infty]$ is the function $f^*:X^*\to (-\infty,\infty]$ defined by $f^*(x^*):=\sup\{\langle x^*,x\rangle-f(x): x\in X\}$ for all $x^*\in X^*$. The biconjugate (or bidual) of $f$ is defined by $f^{**}(x):=\sup\{\langle x^*,x\rangle-f^*(x^*): x^*\in X^*\}$ for all $x\in X$ (of course, we restrict here our attention to $X\cong X^{**}$; in a non-reflexive  Banach  space the definition involves $X^{**}$). The effective domain of $f$ is the set $\dom(f):=\{x\in X: f(x)<\infty\}$ and $f$ is said to be proper whenever $\dom(f)\neq\emptyset$.  The subdifferential of $f$ at $x\in X$ is the set $\partial f(x):=\{x^*\in X^*: f(x)+\langle x^*,w-x\rangle\leq f(w)\,\,\forall w\in X\}$. We say that $f$ is G\^ateaux differentiable at $x\in X$ whenever it is finite at $x$ and there exists a continuous linear functional $\nabla f(x)\in X^*$ such that 
\begin{equation}\label{eq:Gateaux}
\langle \nabla f(x),y\rangle=\lim_{t\to 0}\frac{f(x+ty)-f(x)}{t},\quad \forall\, y\in X.
\end{equation}
We say that $f$ is Fr\'echet differentiable (or simply differentiable) at $x\in X$ if $f(x)\in \R$ and there exists a continuous linear functional $f'_F(x)\in X^*$ such that for all sufficiently small $h\in X$ 
\begin{equation}\label{eq:Frechet}
 f(x+h)=f(x)+\langle f'_F(x),h\rangle+o(\|h\|).
\end{equation}
It is well known that Fr\'echet differentiability implies G\^ateaux differentiability and  conversely, if the G\^ateaux derivative is continuous at a point, then it is Fr\'echet differentiable there (and in both implications these notions coincide) \cite[pp.  13-14]{AmbrosettiProdi1993book}. It is also well known that when $X$ is finite-dimensional and $f$ is lower semicontinuous, convex and proper, then  $f$ is G\^ateaux differentiable at $x\in\dom(f)$ if and only if it is Fr\'echet differentiable there \cite[Corollary 17.44, p. 306]{BauschkeCombettes2017book}, \cite[Theorem 25.2, p. 244]{Rockafellar1970book}.

Now we discuss the definition of resolvent and protoresolvent. 
\begin{defin}\label{def:Resolvent}
Given a real reflexive Banach space $X$, If $f:X\to \R$ is G\^ateaux differentiable on $X$, then the \emph{resolvent} of $A:X\to 2^{X^*}$ relative to $f$ is the operator $\textnormal{Res}^f_{A}: X\to 2^X$ defined by 
\begin{equation}\label{eq:Res^f_A} 
\textnormal{Res}^f_{A}(x):=(\nabla f+ A)^{-1}(\nabla f(x)),\,\,\forall x\in X,
\end{equation} 
and the \emph{protoresolvent} of $A$ relative to $f$ is the operator $\textnormal{prot}^f_{A}:X^*\to 2^X$ defined by 
\begin{equation}\label{eq:protres^f_A} 
\textnormal{prot}^f_{A}(x^*):=(\nabla f+ A)^{-1}(x^*),\,\,\forall x^*\in X^*.
\end{equation} 
\end{defin}
In Definition \bref{def:Resolvent} (and elsewhere) we use the following conventions. First, given an arbitrary $B:X\to 2^{X^*}$, the inverse of $B$ is the operator $B^{-1}:X^*\to 2^X$ defined by $B^{-1}(w):=\{x\in X:  w\in Bx\}$ for all $w\in X^*$ (namely $w\in B(x)$ if and only if $x\in B^{-1}(w)$). Second, given two subsets $S_1$ and $S_2$ (of either $X$ or $X^*$),  their sum is $S_1+S_2:=\{s_1+s_2: s_1\in S_1, s_2\in S_2\}$ if both of them are nonempty and $S_1+S_2:=\emptyset$ otherwise. In particular, if we identify the singleton $S_1:=\{s_1\}$ with $s_1$, then $s_1+S_2=\{s_1+s_2: s_2\in S_2\}$ holds whenever $S_2\neq\emptyset$. Third, given $A,B:X\to 2^{X^*}$ and $x\in X$, we define  $(A+B)x:=Ax+Bx$ (in particular, $(A+B)x\neq\emptyset$ if and only if $Ax\neq \emptyset$ and $Bx\neq\emptyset$). 

It seems that $\Res^f_{A}$ was introduced by Eckstein \cite{Eckstein1993jour} in finite-dimensional Euclidean spaces $X$ for strictly convex functions (actually Bregman) $f$ defined on closed and convex subsets of the space and for monotone operators $A$, but closely related versions of it had been  discussed before by Kassay \cite{Kassay1985jour} and Ha  \cite{Ha1990jour}. Resolvents relative to special functions had been, of course, well known in the literature much  before  \cite{Eckstein1993jour} in various equivalent forms for the case where $X$ is a real Hilbert space, $A=\partial F$ where $F:X\to(-\infty,\infty]$ is lower semicontinuous proper convex function, and $f=c\|\cdot\|^2$ for some $c>0$ (usually $c=1/2$): see, for example, \cite{BruckReich1977jour,Moreau1962jour,Moreau1965jour,Rockafellar1976jour} among many other papers. In this latter case (namely, when $A=\partial F$) the resolvent $\Res^f_{A}$ is frequently denoted by $\textnormal{prox}_f$. An interesting observation which was essentially made in \cite[p. 210]{Eckstein1993jour}  when $A=\partial F$ for some lower semicontinuous proper convex function $F:X\to(-\infty,\infty]$ is that for each $x\in X$ one has $\Res^f_{A}(x)=\textnormal{argmin}_{z\in X}(F(z)+D_f(z,x))$, where $D_f(z,x)$ is the Bregman distance between $z$ and $x$ (see \beqref{eq:D_f} below). This identity generalizes the well-known identity regarding the connection between resolvents and minimization problems, a connection which appears already in 
\cite[p. 2897]{Moreau1962jour} and \cite[p. 278]{Moreau1965jour} in the classical case where $X$ is a real Hilbert space and $f=c\|\cdot\|^2$ for some $c>0$ (see also \cite[p. 455]{CensorZenios1992jour} and \cite[p. 671]{Teboulle1992jour} for versions of this identity related to Bregman distances and other distances). 

A thorough investigation of the resolvent relative to lower semicontinuous and convex functions $f:X\to (-\infty,\infty]$ defined on a general Banach space $X$ and G\^ateaux differentiable in the interior of their effective  domains was carried out by Bauschke, Borwein and Combettes in \cite{BauschkeBorweinCombettes2003jour}, where $\Res^f_{A}$ was called ``$D$-resolvent''. Generalization of this concept (to $F$-resolvents) and further  developments appear in Bauschke, Wang and Yao \cite{BauschkeWangYao2010inbook}. The terminology  ``the resolvent of $A$ relative to $f$'' and the notation $\Res^f_{A}$ first appeared in a paper of Reich and Sabach  \cite{ReichSabach2009jour}, but a closely related terminology appeared  in  G{\'a}rciga Otero and Iusem \cite[Definition 3]{GarcigaOtero-Iusem2007jour}: ``the resolvent of $A$ with respect to a regularization function $f$''.  

We finish this section by noting that in the special but important case where $(X,\langle \cdot,\cdot\rangle)$ is a real Hilbert space there is, of course, a  slightly modified  version of the definitions and results presented in this paper (for instance, Propositions \bref{prop:PrincipleResolvent} and \bref{prop:ImplicitInexactness} below), since, as usual, we identify $X$ and $X^*$ via the natural correspondence coming from the well-known  Riesz-Fr\'echet representation theorem  \cite[Theorem 5.5, p. 135]{Brezis2011book},  redefine $f^*(x^*):=\sup\{\langle  x^*,x\rangle -f(x): x\in X\}$ for all $x^*\in X$, and for each $x\in X$ we identify $\nabla f(x)\in X^*$ with the vector in $X$ coming from the Riesz-Fr\'echet theorem. 

\section{Fully Legendre functions}\label{sec:FullyLegendre}
In this section we introduce the class of \emph{fully Legendre functions} and present some properties and examples related to them. 
\begin{defin}\label{def:FullyLegendre}
Let $(X,\|\cdot\|)$ be a real reflexive Banach space and let $f:X\to(-\infty,\infty]$. If $f$ is lower semicontinuous, convex and G\^ateaux differentiable (hence finite) on $X$ and if $f^*$ is G\^ateaux differentiable on $X^*$, then $f$ is called {\bf fully Legendre}.
\end{defin}
 The class of fully Legendre functions is quite rich and contains numerous mundane functions. For instance, in addition to the functions presented in Examples \bref{ex:PositiveDefinite}--\bref{ex:PowerRho} below, we present in Remark \bref{rem:FullyLegendreFIniteDim} a certain geometric characterization of fully Legendre functions when the space is finite-dimensional. It turns out that this finite-dimensional characterization is equivalent to saying that $f$ is fully Legendre if and only if it is differentiable over the entire space, strictly convex there and super-coercive (namely, $\lim_{\|x\|\to\infty}f(x)/\|x\|=\infty$). Hence many everyday examples  of convex functions, such as the ones shown in Figures \bref{fig:FullyLegendre1D}-\bref{fig:FullyLegendre2D}, are fully Legendre. In Remark \bref{rem:FullyLegendreFIniteDim} below we also  explain why a fully Legendre function defined on a finite-dimensional space must be a Bregman function. Since Bregman functions have numerous applications in optimization, nonlinear analysis, machine learning, compress sensing and elsewhere (see, for example, \cite{BanerjeeMeruguDhillonGhosh2005jour,BauschkeBorweinCombettes2003jour,Bregman1967jour,CensorLent1981jour,Eckstein1993jour,Reem2012incol,ReichSabach2010jour,Teboulle1992jour,YinOsherGoldfarbDarbon2008jour} and the references therein), this fact   increases further the potential of the class of fully Legendre functions. 

As explained in Remark \bref{rem:Legendre} below, fully Legendre functions are a special case of Legendre functions (a notion which was introduced in \cite[Section 26]{Rockafellar1970book} and was extended and thoroughly investigated in \cite{BauschkeBorweinCombettes2001jour}) in which the effective domain of $f$ is the entire space $X$ and the effective domain of $f^*$ is the entire dual space $X^*$. Hence we feel that the terminology ``fully Legendre'' is appropriate. There is, of course, a symmetry between $f$ and $f^*$ in Definition  \bref{def:FullyLegendre} because it is well known \cite[p. 11]{Brezis2011book} that $f^*$ is always convex and lower semicontinuous on $X^*$. This symmetry between the properties of a fully Legendre function and its conjugate is typical: for instance, both of them are strictly convex and their gradients are locally bounded (Remark \bref{rem:Legendre} below). 

\begin{figure}[t]
\begin{minipage}[t]{0.49\textwidth}
\begin{center}{\includegraphics[clip, scale=0.56]{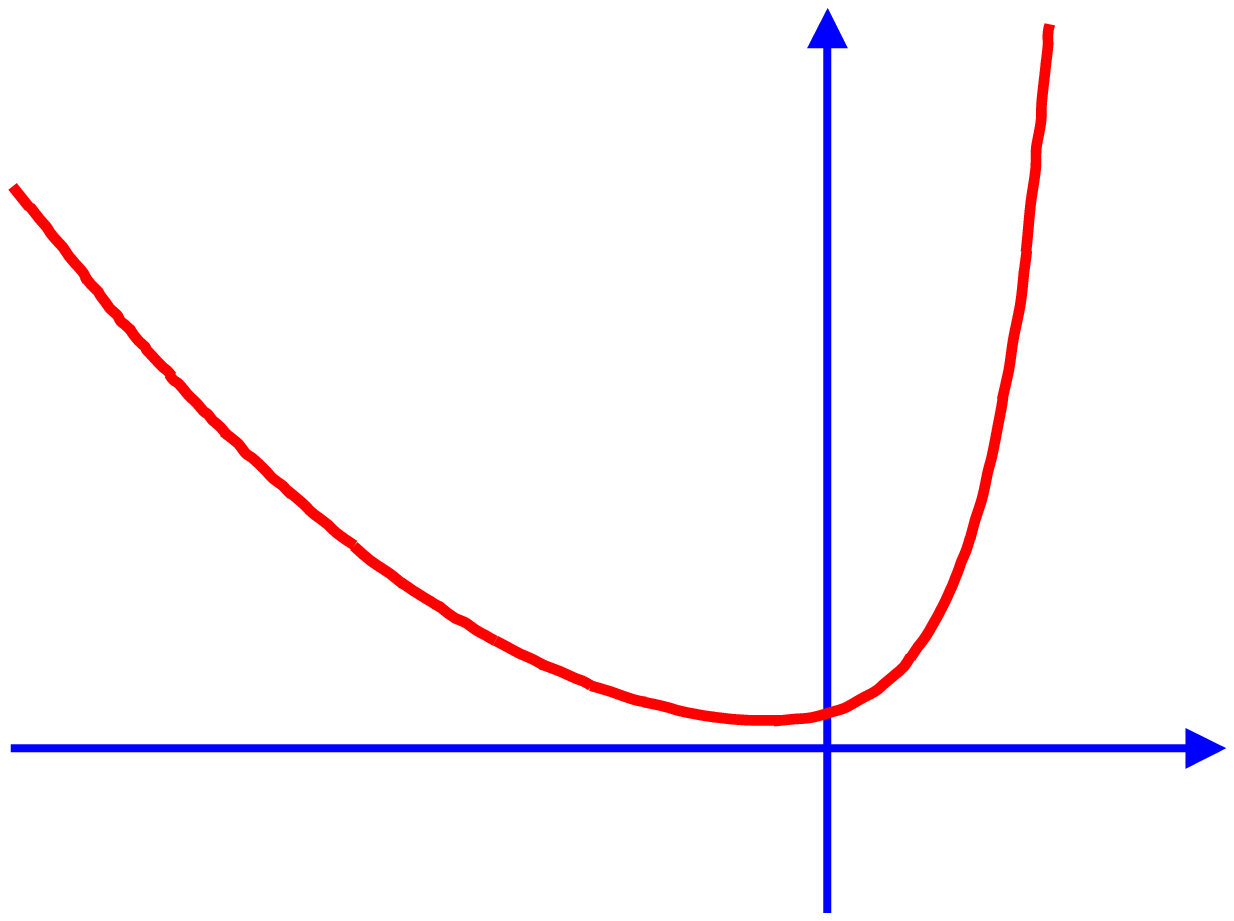}}
\end{center}
 \caption{The graph of a mundane one-dimensional fully Legendre function, based on the characterization  mentioned in Remark \bref{rem:FullyLegendreFIniteDim}.}
\label{fig:FullyLegendre1D}
\end{minipage}
\hfill
\begin{minipage}[t]{0.49\textwidth}
\begin{center}
{\includegraphics[scale=0.41]{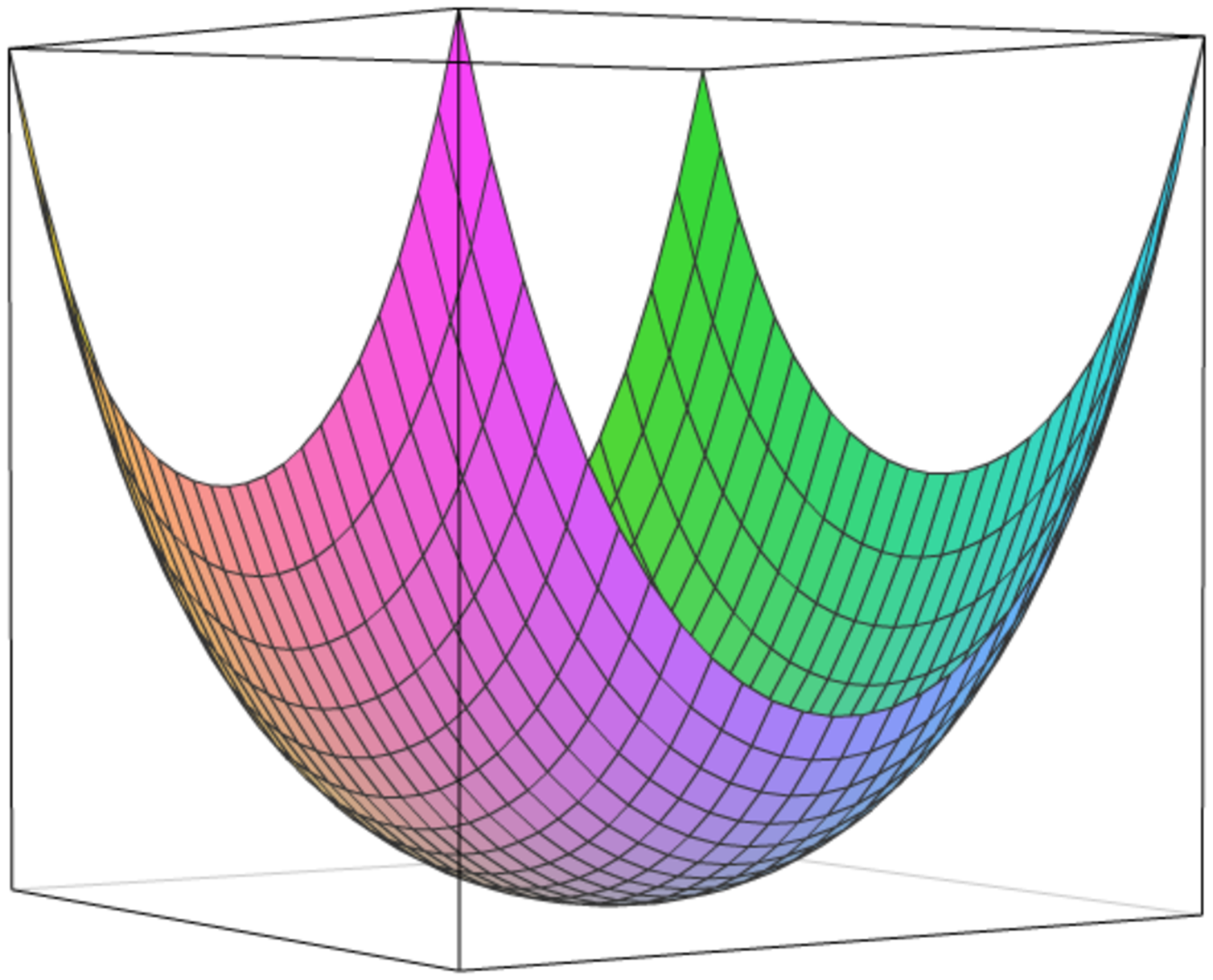}}
\end{center}
 \caption{The graph of a mundane two-dimensional fully Legendre function, based on the characterization  mentioned in Remark \bref{rem:FullyLegendreFIniteDim}.}
\label{fig:FullyLegendre2D}
\end{minipage}
\end{figure}

Here are a few simple examples of fully Legendre functions. Additional examples can be found in \cite[Sections 6, 7]{BauschkeBorweinCombettes2001jour}. We note that using the results mentioned in \cite{BauschkeBorwein1997jour} (e.g., Proposition  5.1, Theorem 5.12, Subsection 7.2), many new examples of fully Legendre functions can be constructed from old ones in the finite-dimensional case.  
\begin{expl}\label{ex:PositiveDefinite}
Let $(X,\langle\cdot,\cdot\rangle)$ be a real Hilbert space and $f(x):=\frac{1}{2}\langle Ax,x\rangle$ for every $x\in X$, where $A:X\to X$ is a continuous, invertible, positive semidefinite and symmetric linear operator. In this case elementary calculations show that $f^*(u)=\frac{1}{2}\langle A^{-1}u,u\rangle$ for each $u\in X$ and we have $\nabla f=A$, $\nabla f^*=A^{-1}$. 
\end{expl}
\begin{expl}\label{ex:cosh}
Let $(X,\|\cdot\|)$ be the finite-dimensional Euclidean space $\R^m$, $m\in\N$ and let $f(x):=\sum_{i=1}^m \cosh(x_i)$, $x=(x_i)_{i=1}^m\in X$. Then, as follows from \cite[p. 50]{BorweinLewis2006book} and an elementary calculation, one has  $f^*(u)=\sum_{i=1}^m \left(u_i\sinh^{-1}(u_i)-\sqrt{1+u_i^2}\right)$, $u=(u_i)_{i=1}^m\in X$. Of course, differentiability of both $f$ and $f^*$ follows from the differentiability of the hyperbolic trigonometric functions and their inverses. 
\end{expl}
\begin{expl}
Suppose that $X=\R^m$ for some $m\in\N$ and $f:X\to \R$ 
is twice continuously differentiable and its Hessian $f''$ 
is positive definite at each point. For each $x\in X$ 
let $f_2(x):=\inf\{\langle f''(x)w,w\rangle: w\in X,  \|w\|=1\}$.  
We claim that if $f''$ satisfies an asymptotically slow decay condition in the sense that there exist $\rho\in (0,1)$, $r>0$ and $\beta>0$ such that 
$f_2(x)\geq \beta/\|x\|^{\rho}$ for all $x\in X$ satisfying 
$\|x\|\geq r$, then $f$ is fully Legendre. In particular (by taking $r=1$ and any $\rho\in (0,1)$), 
if $f:X\to\R$ is twice continuously differentiable and 
its Hessian is strongly positive definite in the sense 
that for some $\beta>0$ we have $\inf_{x\in X}f_2(x)\geq \beta$, 
then $f$ is fully Legendre.  

Before proving the assertion, here are two illustrations of functions which satisfy the above-mentioned conditions. First, let $X=\R$ and let $f:X\to\R$ be defined by $f(x):=x^{1.5}-(3/4)x+(1/8)$  when $x\geq 1$ and $f(x):=3x^2/8$ when $x\leq 1$. Then $f$ is twice continuously differentiable on $X$. In addition, for all $x\in X$ satisfying $x\geq 1$ the Hessian of $f$ at $x$ is  $(3/4)x^{-0.5}$ and for $x\leq 1$ the Hessian is $3/4$. Hence the Hessian of $f$ is positive definite and $f_2(x)=\inf\{\langle f''(x)w,w\rangle: |w|=1\}=(3/4)x^{-0.5}$ whenever $x\geq 1$ and $f_2(x)=3/4$ when $x\leq 1$, namely the Hessian of $f$ satisfies the asymptotically slow decay condition with $\rho=0.5$, $r=1$, and $\beta=3/4$. As an illustration of the strongly positive Hessian condition, let $f(x):=\sum_{i=1}^m \cosh(x_i)$, $x=(x_i)_{i=1}^m\in X=\R^m$, $m\in\N$ be as in Example \bref{ex:cosh}. For all $x\in X$ the Hessian of $f$ at $x$ is the diagonal matrix the entries of which are $\cosh(x_i)$, $i\in \{1,\ldots,m\}$. Hence $\langle f''(x)w,w\rangle=\sum_{i=1}^m\cosh(x_i)w_i^2\geq \sum_{i=1}^m w_i^2=1$ for every $x\in X$ and $w\in X$ satisfying $\|w\|=1$, since $\cosh(t)\geq 1$ for all $t\in \R$.

Now we prove the assertion. Since $f''(x)$ is positive definite at each $x\in X$, a  
well-known result ensures that $f$ is strictly convex \cite[Theorem 4.3.1(ii), p. 115]{Hiriart-UrrutyLemarechal2001book}. Now, since $f$ is obviously differentiable on $X$, in order to 
see that $f$ is fully Legendre it remains to show, according 
to the characterization of finite-dimensional fully Legendre 
functions (Remark \bref{rem:FullyLegendreFIniteDim} below), that $f$ is super-coercive. 
Fix some $0\neq x\in X$. The Taylor expansion of $f$ of order 1 
about 0 with a remainder in Lagrange's form ensures that for 
some $y_x$ located strictly inside the line segment which connects 0 and $x$ 
we have 
\begin{multline}\label{eq:Taylor0}
f(x)=f(0)+\langle\nabla f(0),x\rangle+\frac{1}{2}\langle f''(y_x)x,x\rangle\\
=\|x\|^2\frac{1}{2}\left\langle f''(y_x)\frac{x}{\|x\|},\frac{x}{\|x\|}\right\rangle+f(0)+\langle\nabla f(0),x\rangle.
\end{multline} 
Hence, if $\|x\|\geq r$, then \beqref{eq:Taylor0}, our assumption on $f_2(x)$,  the Cauchy-Schwarz inequality, the fact that 
$0<\|y_x\|<\|x\|$, and the fact that $\rho\in (0,1)$, all imply that indeed $f$ is super-coercive:
\begin{multline*}
\frac{f(x)}{\|x\|}\geq \frac{0.5\|x\|^2f_2(y_x)-|f(0)|-|\langle\nabla f(0),x\rangle|}{\|x\|}\geq 
\|x\|\frac{0.5\beta}{\|y_x\|^{\rho}}-\frac{|f(0)|}{\|x\|}-\frac{\|\nabla f(0)\|\|x\|}{\|x\|}\\
\geq \frac{\beta}{2}\|x\|^{1-\rho}\left(\frac{\|x\|}{\|y_x\|}\right)^{\rho}-\frac{|f(0)|}{\|x\|}-\|\nabla f(0)\|
\geq \frac{\beta}{2}\|x\|^{1-\rho}-\frac{|f(0)|}{\|x\|}-\|\nabla f(0)\|\xrightarrow[\|x\|\to\infty]{}\infty.
\end{multline*}
\end{expl}

\begin{expl}\label{ex:PowerRho}
Suppose that $(X,\|\cdot\|)$ is a smooth and strictly convex (rotund) real Banach space and let 
 $f(x):=(1/\rho)\|x\|^{\rho}$ for a fixed $\rho>1$ and for all $x\in X$. Then, as is well known and follows from elementary calculations, $f^*(u)=(1/\rho^*)\|u\|_*^{\rho^*}$ for all $u\in X^*$,  where $\rho^*>1$ is the dual of $\rho$, namely $(1/\rho)+(1/\rho^*)=1$. It follows from \cite[Lemma 6.2]{BauschkeBorweinCombettes2001jour} that $f$ is fully Legendre and smooth. In fact, if, in addition, $(X,\|\cdot\|)$ is uniformly convex and uniformly smooth, then $f$ is uniformly convex on closed balls and totally convex \cite[Example 6.5]{BauschkeBorweinCombettes2001jour}.
\end{expl}

The following lemma (Lemma \bref{lem:f^*f^-1}  below) is fundamental and various versions of it are well known (e.g., a certain finite-dimensional version \cite[Theorem 26.5, p. 258]{Rockafellar1970book}). It is called the ``Legendre identity'' in \cite{Polyak2016inbook}. Before presenting its proof, we want to say a few words regarding its importance in the context of our paper. First, Lemma \bref{lem:f^*f^-1} plays an essential role in the proof Lemma \bref{lem:ResolventOrdinaryFunction}; this latter lemma is  essential to the proof of Proposition \bref{prop:PrincipleResolvent}, which by itself is essential to the proof of Proposition \bref{prop:ImplicitInexactness}; both Proposition \bref{prop:PrincipleResolvent} and Proposition \bref{prop:ImplicitInexactness} are essential for proving that many known inexact algorithms aiming at finding zeros of set-valued operators are well defined, as shown in Sections \bref{sec:Eckstein}--\bref{sec:ManyMore} below. Second, in many of the above-mentioned inexact algorithms, among them \cite[Algorithm IPPM: Inexact Proximal Point Method, p. 234]{BurachikIusem2008book}, \cite[Inexact Proximal 
Point-Extragradient Method (pp. 75--76)]{Garciga-OteroIusem2004jour},  
\cite[Algorithms II, PI, PII]{IusemGarciga-Otero2002jour}, \cite[Algorithms I, II, IV
]{IusemOtero2001jour}, \cite[Algorithm 1]{SolodovSvaiter2000jour}, it is either  explicitly or implicitly assumed that $\nabla f$ is invertible (even if one restricts the attention to exact algorithms), and this property is needed in the definition of the algorithms themselves; however, either very few sufficient conditions or no ones at all were given in the corresponding works regarding how to find such a function $f$ (which should satisfy additional properties), and it seems that fully Legendre functions are good candidates to be used in these schemes. 
\begin{lem}\label{lem:f^*f^-1} 
If $(X,\|\cdot\|)$ is a real reflexive Banach space and $f:X\to\R$ is fully Legendre, then $\nabla f$ is invertible and 
\begin{equation}\label{eq:f^*f^-1}
\nabla f^*=(\nabla f)^{-1}.
\end{equation}
\end{lem}
\begin{proof}
Since $X$ is reflexive and  because  $f$ is lower semicontinuous, proper (actually finite) and convex function as a fully Legendre function, it follows from \cite[pp. 13 and 67]{Brezis2011book} that $f^{**}=f$. Therefore it can be concluded from \cite[p. 211]{Rockafellar1970jour} or \cite[p. 83]{BonnansShapiro2000book} that for each  $x\in X$ and $x^*\in X^*$ one has $x^*\in \partial f(x)$ if and only if $x\in \partial f^*(x^*)$. Since the definition of the inverse operator implies that $x^*\in \partial f(x)$ if and only if $x\in (\partial f)^{-1}(x^*)$, one has $(\partial f)^{-1}=\partial f^*$. However, since \cite[Theorem 5.37, p. 77]{VanTiel1984book} implies that the subgradient of a G\^ateaux differentiable convex function coincides with  the singleton containing the gradient of the function and because both $f$ and $f^*$ are G\^ateaux differentiable on $X$ and $X^*$ respectively, the above discussion implies that $\nabla f:X\to X^*$ is invertible and \beqref{eq:f^*f^-1} holds, as claimed.
\end{proof}

The next corollary is nothing but a simple observation. We mention it because there is another notion of resolvent, called ``the conjugate resolvent'' \cite[Definition  5.1]{Martin-MarquezReichSabach2012jour}: this is the operator  $\textnormal{CRes}^f_A:X^*\to 2^{X^*}$  defined by $\textnormal{CRes}^f_A:=(I+A\circ \nabla f^*)^{-1}$. When $f$ is fully Legendre, then we can use Lemma \bref{lem:f^*f^-1} to conclude that  $\textnormal{CRes}^f_A=((\nabla f+A)\circ \nabla f^*)^{-1}$, and hence Corollary \bref{cor:ReolventComposition} below implies that the conjugate resolvent is a certain permutation of the resolvent. 
\begin{cor}\label{cor:ReolventComposition}
Under the assumptions of Lemma \bref{lem:f^*f^-1} we have
\begin{equation}\label{eq:ResInverse}
\Res^f_{A}=(\nabla f+ A)^{-1}\circ(\nabla f^*)^{-1}=\left(\nabla f^*\circ(\nabla f+ A)\right)^{-1}.
\end{equation}
\end{cor}

The following assertion describes a fundamental property of the resolvent and the protoresolvent (Definition \bref{def:Resolvent} above). Part \beqref{item:SingleValued} of it is implicit in \cite{BauschkeBorweinCombettes2003jour, BauschkeWangYao2010inbook}  and was mentioned in \cite{ReichSabach2009jour,ReichSabach2010b-jour} without a proof.  For the sake of completeness, we do present the proof below, but before presenting it we note that given a set-valued operator $B:X\to 2^Y$ between two nonempty sets $X$ and $Y$, a necessary and sufficient condition for its inverse $B^{-1}$ to be single-valued is that $\cup_{x\in X}Bx=Y$ and $B(x)\cap B(x')=\emptyset$ for all $x,x'\in X$ satisfying $x\neq x'$.

\begin{lem}\label{lem:ResolventOrdinaryFunction}
Let $(X,\|\cdot\|)$ be a real reflexive Banach space, let $A:X\to 2^{X^*}$ be maximally monotone, $f:X\to\R$ be fully Legendre, and let $\lambda>0$. Then 
\begin{enumerate}[(i)]
\item\label{item:SingleValued} $\textnormal{Res}^f_{\lambda A}$ and $(\nabla f+\lambda A)^{-1}$ are  single-valued. In particular,  $\dom(\textnormal{Res}^f_{\lambda A})=X$ and $\dom((\nabla f+\lambda A)^{-1})=X^*$. 
\item\label{item:MaximallyMonotone}  $(\nabla f +\lambda A)^{-1}$ is maximally monotone. 
\end{enumerate}
\end{lem} 

\begin{proof}
We start by presenting two proofs of Part \beqref{item:SingleValued}. \\

{\bf \noindent Way 1:} Since $A$ is monotone, $\lambda>0$ and since $f$ is fully Legendre and hence strictly convex, it follows from  \cite[Proposition  3.8(iv)(b)]{BauschkeBorweinCombettes2003jour} that $\textnormal{Res}^f_{\lambda A}$ is single-valued on its effective domain. Since $A$ is maximally monotone and hence nontrivial, since $X$ is reflexive, and since the range of $\nabla f$ is $X^*$ by Lemma \bref{lem:f^*f^-1}, we can use \cite[Theorem 3.13(iv)(b)]{BauschkeBorweinCombettes2003jour} (where the meaning of $f$ being cofinite is that $\dom(f^*)=X^*$, a condition which is fulfilled in our case since $f$ is fully Legendre) from which it follows that the resolvent belongs to the class of operators $\mathfrak{B}$ defined in  \cite[Definition 3.1]{BauschkeBorweinCombettes2003jour}. Since, according to the definition of $\mathfrak{B}$, the  effective domain of each operator which belongs to it is equal to the interior of the effective  domain of $f$, and since $f$ is defined on $X$ and its range is $\R$, we conclude that $\dom(\textnormal{Res}^f_{\lambda A})=X$. This fact, when combined with the first lines of the proof, imply that  $\textnormal{Res}^f_{\lambda A}$ is single-valued on $X$. Since $(\nabla f+ \lambda A)^{-1}=\textnormal{Res}^f_{\lambda A}\circ (\nabla f^*)$, it follows that $(\nabla f+ \lambda A)^{-1}$ is a composition of two single-valued operators and therefore it is single-valued too. \\

{\bf \noindent Way 2:} Let $F:=\{\nabla f\}$. Then $F$ is single-valued. Since $f$ is convex and G\^ateaux differentiable we have $F=\partial f$ according to \cite[Theorem 5.37, p. 77]{VanTiel1984book}. Thus  Rockafellar's  theorem \cite[Theorem A]{Rockafellar1970jour} implies that $F$ is maximally monotone. In  addition, $F$ is strictly monotone (since $f$ is strictly convex), $3^*$ monotone \cite[Lemma  3.10(iv)]{BauschkeBorweinCombettes2003jour}  and surjective (Lemma \bref{lem:f^*f^-1}). Since $A$ is maximally monotone, the above discussion implies, according to \cite[Proposition 4.2(iv)]{BauschkeWangYao2010inbook}, that the effective domain of the $F$-resolvent of $A$ is $X$. But the $F$-resolvent of $A$ is $(A+F)^{-1}\circ F$ (see  \cite[Definition  4.1]{BauschkeWangYao2010inbook}), namely it coincides with $\textnormal{Res}^f_{\lambda A}$. Therefore $\dom(\textnormal{Res}^f_{\lambda A})=X$ and hence  $\textnormal{Res}^f_{\lambda A}(x)$ contains  at least one element from $X$ for each $x\in X$. In addition, since $A$ is  monotone we can use \cite[Proposition 4.2(iii)]{BauschkeWangYao2010inbook} to deduce that 
 $\textnormal{Res}^f_{\lambda A}(x)$ contains at most one element from $X$ for each $x\in X$. Thus $\textnormal{Res}^f_{\lambda A}$ is single-valued and hence so is $(\nabla f+ \lambda A)^{-1}=\textnormal{Res}^f_{\lambda A}\circ (\nabla f^*)$. \\
 
Consider now Part \beqref{item:MaximallyMonotone}. From \cite[Proposition 3.12]{BauschkeBorweinCombettes2003jour} it follows that $\nabla f +\lambda A$ is maximally monotone. Since it is straightforward to check that an operator is maximally monotone if and only if its inverse is maximally monotone, it follows that  $(\nabla f +\lambda A)^{-1}$ is maximally monotone. 
\end{proof}

\begin{remark}\label{rem:Legendre} 
In  \cite{BauschkeBorweinCombettes2001jour} there is a general discussion concerning Legendre functions. There $X$ is an arbitrary Banach space and a proper lower  semicontinuous and convex function $f:X\to(-\infty,\infty]$ is called Legendre if it is both essentially smooth (meaning that $\partial f$ is both locally bounded and single-valued on its effective domain) and essentially strictly convex (namely, $f$ is strictly convex on every convex subset of $\dom(\partial f)$ and $(\partial f)^{-1}$ is locally bounded on its effective domain). 

If we assume that $\dom(f)=X$, then we can use \cite[Theorem  5.6(iv)]{BauschkeBorweinCombettes2001jour} to conclude that $f$ is essentially smooth if and only if it is G\^ateaux differentiable. Thus if both $f$ and $f^*$ are  G\^ateaux differentiable on $X$ and $X^*$, respectively, then both of them are essentially smooth. If we also assume that $X$ is reflexive, then we can use \cite[Theorem  5.4]{BauschkeBorweinCombettes2001jour} to conclude that both $f^*$ and $f^{**}$ are essentially strictly convex. Since in reflexive spaces  we have $f^{**}=f$ for each lower semicontinuous, proper and convex function $f:X\to (-\infty,\infty]$ (see, for instance, \cite[pp. 13 and 67]{Brezis2011book}), it follows that both $f^*$ and $f$ are essentially strictly convex. Thus  both $f$ and $f^*$ are Legendre functions. We conclude from the above-mentioned discussion and Definition \bref{def:FullyLegendre} above that a fully Legendre function is indeed a Legendre function. On the other hand, if $f$ is Legendre and the space is reflexive, then $f^*$ is also a Legendre function  \cite[Corollary 5.5]{BauschkeBorweinCombettes2001jour}. Therefore both $f$ and $f^*$ are essentially  smooth. If we also assume that  both of them are finite, then both functions are G\^ateaux differentiable according to  \cite[Theorem 5.6(iv)]{BauschkeBorweinCombettes2001jour}. The above discussion shows that if $X$ is a real reflexive Banach space and both $f$ and  $f^*$ are finite on $X$ and $X^*$ respectively, then $f$ is Legendre if and only both $f$ and $f^*$ are G\^ateaux differentiable on $X$ and $X^*$ respectively. 
\end{remark}

\begin{remark}\label{rem:FullyLegendreFIniteDim}
If our space $(X,\|\cdot\|)$ is $\R^m$ with the Euclidean norm (or any other norm) for some $m\in\N$, then there is a certain geometric characterization  for a function $f:X\to(-\infty,\infty]$ to be fully Legendre, a characterization which is perhaps more illuminating than Definition \bref{def:FullyLegendre}. Moreover, this characterization is equivalent to a simple and useful condition which involves the notion of super-coercive   functions. Using this latter condition, we explained below why a fully Legendre function defined on a finite-dimensional space must be a Bregman function. 

In order to derive these results, we recall that, according to Rockafellar \cite[p. 66]{Rockafellar1970book}, the recession function associated with a proper lower semicontinuous convex function $f:X\to(-\infty,\infty]$ is the function  $f_{\infty}:X\to(-\infty,\infty]$ which is determined by any of the following two identities: 
\begin{subequations}\label{eq:f_infty}
\begin{align}
f_{\infty}(z)&=\sup\{f(z+w)-f(w): w\in\dom(f)\}, \quad \forall z\in X,\\
f_{\infty}(z)&=\lim_{\lambda\to\infty}\frac{f(w+\lambda z)-f(w)}{\lambda}=\lim_{\lambda\to\infty}\frac{f(w+\lambda z)}{\lambda}, \quad \forall z\in X, \forall w\in\dom(f).
\end{align}
\end{subequations}
Here we follow the notation in Auslender and Teboulle \cite[p. 48 and elsewhere]{AuslenderTeboulle2003book} who  call $f_{\infty}$ ``the asymptotic function''. Rockafellar \cite[p. 66 and elsewhere]{Rockafellar1970book} denotes the recession function by ``$f0^+$''. Various properties, examples and applications of $f_{\infty}$ can be found in the books of Rockafellar \cite{Rockafellar1970book} and Auslender and Teboulle \cite{AuslenderTeboulle2003book}, in particular \beqref{eq:f_infty} which appears in \cite[Proposition 2.5.2, p. 50]{AuslenderTeboulle2003book}.  

As shown in the next paragraph, \emph{because $X$ is finite-dimensional, $f$ is fully Legendre if and only if it has the following properties: it is (Fr\'echet) differentiable (hence finite) on $X$, it is strictly convex there, and its recession function satisfies $f_{\infty}(z)=\infty$ for all $0\neq z\in X$}. But according to 
\cite[Proposition 2.16]{BauschkeBorwein1997jour}, if $f:X\to(-\infty,\infty]$ is a lower semicontinuous proper and convex function, then $f_{\infty}(z)=\infty$ for all $0\neq z\in X$ if and only if $f$ is super-coercive, namely $\lim_{\|u\|\to\infty}f(u)/\|u\|=\infty$. {\bf Thus a simple and useful equivalent condition for $f:X\to(-\infty,\infty]$ to be fully Legendre when the space $X$ is finite-dimensional is as follows: $f$ is (Fr\'echet) differentiable on the whole space, strictly convex there and super-coercive}. Figures \bref{fig:FullyLegendre1D}-\bref{fig:FullyLegendre2D} above present typical examples of functions having these properties. We note that as observed in Rockafellar \cite[p. 259]{Rockafellar1970book}, a finite convex function satisfies the condition $f_{\infty}(z)=\infty$ for all $0\neq z\in X$ if and only if its epigraph $\{(x,t)\in X\times \R: t\geq f(x)\}$ does not contain any non-vertical half-line (Rockafellar \cite[p. 259]{Rockafellar1970book} calls a finite convex function which satisfies the latter condition ``co-finite''). 

Now we prove the above-mentioned characterization. Suppose first that $f$ is fully Legendre. Then $f$ and $f^*$ are G\^ateaux differentiable on $X$ and hence, by definition, finite there. Since $X$ is finite-dimensional and both of them are convex, both of them are Fr\'echet differentiable there according to Rockafellar \cite[Theorem 25.2, p. 242]{Rockafellar1970book}. Since $X$ is reflexive, we can conclude from Remark \bref{rem:Legendre} above that both $f$ and $f^*$ are strictly convex. As a result, we can apply Rockafellar \cite[Theorem 26.6, p. 259]{Rockafellar1970book} and Lemma \bref{lem:f^*f^-1} above to conclude that $f_{\infty}(z)=\infty$ for all $0\neq z\in X$. Conversely, if $f$ is Fr\'echet differentiable on $X$, strictly convex there, and $f_{\infty}(z)=\infty$ for all $0\neq z\in X$, then in particular $f$ is a convex and lower semicontinuous (actually continuous) G\^ateaux differentiable (thus finite) function. Moreover, \cite[Theorem 26.6, p. 259]{Rockafellar1970book} implies that $f^*$ is Fr\'echet (thus G\^ateaux) differentiable on $X$. Hence we can use Definition \bref{def:FullyLegendre} to conclude that $f$ is fully Legendre. 

Finally, we need to show that when the space $X$ is finite dimensional and $f:X\to\R$ is a fully Legendre function, then $f$ is a Bregman function, namely it satisfies \cite[Definition 4.1]{BauschkeBorwein1997jour}. Indeed, the above-mentioned characterization implies that $f$ is strictly convex and differentiable on $X$ and that it is super-coercive. Hence we can use \cite[Corollary 4.8]{BauschkeBorwein1997jour} to conclude that $f$ is a Bregman function. 
\end{remark}

\begin{remark}\label{lem:FormulaConjugate}
Given a real reflexive Banach space $X$, a fully Legendre function $f:X\to\R$, and $x^*\in X^*$, it is possible to express $f^*(x^*)$ in an explicit manner, as done in \beqref{eq:f^*x^*} below. Indeed, consider the function $F:X\to\R$ defined by $F(x):=f(x)-\langle x^*,x\rangle$ for each  
$x\in X$. This function is  proper, lower semicontinuous, convex, and G\^ateaux differentiable on $X$ and hence (see \cite[Theorem 5.37, p. 77]{VanTiel1984book}) $\nabla F(x)=(\partial F)(x)$.  Moreover, since $\nabla F(x)=\nabla f(x)-x^*$, it follows from Lemma \bref{lem:f^*f^-1} that $\nabla F$ has a unique zero and this is the point $x(x^*):=(\nabla f)^{-1}(x^*)$. Thus the well-known characterization of a global minimizer \cite[Theorem 2.5.7, p. 105]{Zalinescu2002book} implies that $x(x^*)$ is a global minimizer of $F$. We conclude from 
the previous discussion and the definition of $f^*$ that $f^*(x^*)=\sup_{x\in X}[-F(x)]=-F(x(x^*))$. In other words,
\begin{equation}\label{eq:f^*x^*}
f^*(x^*)=\langle x^*,(\nabla f)^{-1}(x^*)\rangle-f((\nabla f)^{-1}(x^*)).
\end{equation}
This expression extends, to an infinite-dimensional setting, a similar expression presented in \cite[p. 259]{Rockafellar1970book}. Formula \beqref{eq:f^*x^*} is intimately related to the classical Legendre transform 
from classical mechanics, for sufficiently smooth functions defined on finite-dimensional 
spaces, and it has an application in the theory of fixed points of Legendre-Fenchel 
type transforms: see \cite[Remark 15.2]{IusemReemReich2017prep}.
\end{remark}

\section{The inexact resolvent inclusion problem}\label{sec:PrincipleResolvent}
In this section we present and solve the ``inexact resolvent inclusion problem'', a problem which is, as will be shown in later sections, very relevant to many inexact algorithms aiming at solving \beqref{eq:ZeroA}. Our existence and uniqueness result is presented in Proposition \bref{prop:PrincipleResolvent} below, and a number of comments (Remarks \bref{rem: eta in A}-\bref{rem:HistoryInexactResolventInclusionProblem} below) follow it and clarify certain issues related to it. Neither Proposition \bref{prop:PrincipleResolvent} nor its proof are complicated, and yet, this proposition is quite useful, partly because of its generality. But first, let us formulate the problem (in the formulation below we restrict ourselves to the main setting of this paper, but of course, the problem can be formulated in a wider generality, for instance one can let $X$ to be any normed space and to take $f:X\to\R$ to be any G\^ateaux differentiable function). 

Let $(X,\|\cdot\|)$ be a real reflexive Banach space, let $x\in X$, $\lambda>0$, let $f:X\to\R$ be fully Legendre, let $\eta\in X^*$ and let $A:X\to 2^{X^*}$ be maximally monotone. The \emph{inexact resolvent inclusion problem} is to find $y\in X$ such that 
\begin{equation}\label{eq:PerturbedResolventProblem}
\eta\in Ay+\frac{1}{\lambda}\left(\nabla f(y)-\nabla f(x)\right).
\end{equation}
Another name to this problem is ``the generalized proximal point subproblem'' \cite[p. 215]{SolodovSvaiter2000jour}). The vector $\eta$ can be regarded as being an error term or a perturbation, and although one knows that $\eta$ exists, one may not  necessarily be able to evaluate it (frequently one can only ensure that $\|\eta\|$ is sufficiently small; in this connection, see Remark \bref{rem: eta is unknown} below). When $\eta=0$, then one speaks of the ``exact resolvent inclusion problem''. 

\begin{lem}\label{lem:EquivalenceSystemResolvent}
For each $(x,\eta)\in X\times X^*$ and $\lambda>0$, the point $y\in X$ satisfies \beqref{eq:PerturbedResolventProblem} if and only if there exists $\xi\in X^*$ such that the pair $(y,\xi)$ satisfies the following two conditions: 
\begin{subequations}\label{eq:PrincipleSystem}
\begin{align}
\xi &\in A(y)\label{eq:xi_inclusion},\\
\eta&=\xi+\frac{1}{\lambda}\left(\nabla f(y)-\nabla f(x)\right).\label{eq:eta_xi_y}
\end{align}
\end{subequations}
\end{lem}

\begin{proof}
If some $y\in X$ satisfies \beqref{eq:PerturbedResolventProblem}, then $\eta\in (1/\lambda)(\nabla f(y)-\nabla f(x))+Ay$, so the sum is nonempty and by its definition there exists $\xi\in Ay$ such that $\eta=(1/\lambda)(\nabla f(y)-\nabla f(x))+\xi$, namely \beqref{eq:PrincipleSystem} holds. On the other hand, if \beqref{eq:PrincipleSystem} holds, then the sum $(1/\lambda)(\nabla f(y)-\nabla f(x))+Ay$ is nonempty and $\eta\in (1/\lambda)(\nabla f(y)-\nabla f(x))+Ay$. 
\end{proof}

\begin{prop}\label{prop:PrincipleResolvent}
Let $(X,\|\cdot\|)$ be a real reflexive Banach space and suppose that $f:X\to\R$ is fully Legendre.  Let $A:X\to 2^{X^*}$ be a  maximally monotone operator. Then for all $x\in X$, $\eta\in X^*$ and $\lambda>0$, there exists a unique $y\in X$ such that \beqref{eq:PerturbedResolventProblem} holds and a unique pair $(y,\xi)\in X\times X^*$  such that \beqref{eq:PrincipleSystem} holds. Moreover, the following relations hold:
\begin{subequations}\label{eq:PrincipleExplicit}
\begin{align}
y&=(\nabla f+ \lambda A)^{-1}\left(\lambda\eta+\nabla f(x)\right),\label{eq:y_explicit}\\
\xi&=\eta-\frac{1}{\lambda}(\nabla f(y)-\nabla f(x)).\label{eq:xi_explicit}
\end{align}
\end{subequations}
\end{prop}

\begin{proof}
By Lemma \bref{lem:EquivalenceSystemResolvent}, it is sufficient to show the existence and uniqueness of a pair $(y,\xi)\in X\times X^*$ which satisfies \beqref{eq:PrincipleSystem}. 
We first prove the existence of such a pair. Lemma \bref{lem:ResolventOrdinaryFunction}\beqref{item:SingleValued} ensures that $(\nabla f+ \lambda A)^{-1}$  is single-valued. Hence for all $x\in X$, $\eta\in X^*$ and $\lambda>0$, if we let $y$ to be defined as the right-hand side of \beqref{eq:y_explicit}, then $y$ is well defined. Thus if we define $\xi$ by the right-hand side of  \beqref{eq:xi_explicit}, then $\xi$ is well defined and \beqref{eq:eta_xi_y} holds. It remains to prove \beqref{eq:xi_inclusion}. Since \beqref{eq:xi_explicit} implies that 
\begin{equation}\label{eq:lambda+y=lambda+x}
\nabla f(y)+\lambda\xi=\nabla f(x)+\lambda\eta,
\end{equation}
the following implicit (fixed point) relation follows from \beqref{eq:y_explicit} and \beqref{eq:lambda+y=lambda+x}:  
\begin{equation}
y=(\nabla f+\lambda A)^{-1}(\lambda\xi+\nabla f(y)).
\end{equation}
This equality implies that $y\in(\nabla f+\lambda A)^{-1}(\lambda\xi+\nabla f(y))$ (of course,  $y$ is the unique element in this set). Hence from the definition of the inverse operator 
we see that $\lambda\xi+\nabla f(y)\in (\nabla f+\lambda A)(y)$. 
Since $\nabla f$ is single-valued and since the sum of two sets is nonempty if and only if both sets are nonempty, the above discussion shows the existence of an element $q\in A(y)$ such that 
$\lambda\xi+\nabla f(y)=\nabla f(y)+\lambda q$. Since $\lambda\neq 0$, we conclude that $\xi=q$  and hence \beqref{eq:xi_inclusion} holds. 

Now we prove the uniqueness of a solution to \beqref{eq:PrincipleSystem}. Let $(y,\xi)\in X\times X^*$ be an arbitrary solution to \beqref{eq:PrincipleSystem}. It follows from \beqref{eq:eta_xi_y} that $\xi$ coincides with the right-hand side of \beqref{eq:xi_explicit}. In order to show that $y$ coincides with the 
right-hand side of \beqref{eq:y_explicit}, consider   \beqref{eq:eta_xi_y}. This equality implies \beqref{eq:lambda+y=lambda+x}. By  \beqref{eq:lambda+y=lambda+x} and  \beqref{eq:xi_inclusion} we have $\lambda\eta+\nabla f(x)\in (\nabla f+\lambda A)(y)$. This relation is equivalent to the relation $y\in (\nabla f+\lambda A)^{-1}(\lambda\eta+\nabla f(x))$. Since we know from  Lemma  \bref{lem:ResolventOrdinaryFunction}\beqref{item:SingleValued} that $(\nabla f+ \lambda A)^{-1}$ is single-valued, it follows that $y=(\nabla f+\lambda A)^{-1}(\lambda\eta+\nabla f(x))$, that is, $y$ coincides with the right-hand side of \beqref{eq:y_explicit} and we have uniqueness, as claimed.
\end{proof}

\begin{remark}\label{rem: eta in A}
 It is possible to formulate and prove Proposition \bref{prop:PrincipleResolvent} by embedding the error term $\eta$ inside the operator $A$ (which will be re-defined), but we feel that the current statement and proof better emphasize the presence of the error term. Many of the existing inexact algorithmic schemes cited in our paper support this point of view. We also note that as far as we understand, the proof of Proposition \bref{prop:PrincipleResolvent} (via Lemma \bref{lem:ResolventOrdinaryFunction} above) does not follow directly from \cite[Proposition  3.8]{BauschkeBorweinCombettes2003jour}, but requires  additional tools such as \cite[Theorem 3.13(iv)(b)]{BauschkeBorweinCombettes2003jour}, as done in Lemma \bref{lem:ResolventOrdinaryFunction}; note that we did not assume in that lemma that $A$ has a zero: in this latter case we could use \cite[Corollary 3.14]{BauschkeBorweinCombettes2003jour} instead of \cite[Theorem 3.13(iv)(b)]{BauschkeBorweinCombettes2003jour}. 
\end{remark}
\begin{remark}\label{rem: eta is unknown}
We emphasize again that in Proposition \bref{prop:PrincipleResolvent} above one may or may not be able to evaluate the error term 
$\eta$ (frequently only the magnitude of $\eta$ can be estimated). In particular, in applications usually $\eta$ is not given in advance to the users, but rather appears due to noise or computational errors, and what one knows is simply that $\eta$ exists. But this lack 
of ability to evaluate $\eta$ does not change the assertion proved in Proposition \bref{prop:PrincipleResolvent} that the unknown $y$ can be represented using $\eta$ and other parameters/unknowns which appear in the statement of Proposition \bref{prop:PrincipleResolvent}. This situation is analogous to the case of a simpler relation, for example the equation $2a+3b+c=-1$, in which, even
if all of the involved variables are unknown to the users (for instance because they are random variables which model some noise), it is still possible to represent each one of these unknowns in terms of the other unknowns. 
\end{remark}
\begin{remark}\label{rem:ExplicitSolution}
One may argue that the formula for $y$ given in \beqref{eq:y_explicit} is not really explicit because the computation of the protoresolvent $(\nabla f+\lambda A)^{-1}$ is generally not easy. We agree that the computation of the protoresolvent can be difficult, but we believe that the representation given in \beqref{eq:y_explicit} has advantages. These advantages are illustrated in the continuity results mentioned in Section \bref{sec:Continuity} below, in the strongly implicit version of the inexact resolvent inclusion problem (Proposition \bref{prop:ImplicitInexactness} below), and in the various consequences  of Proposition \bref{prop:ImplicitInexactness} (Sections \bref{sec:SolodovSvaiter1999-1}--\bref{sec:ManyMore} below). 
\end{remark}

\begin{remark}\label{rem:HistoryInexactResolventInclusionProblem}
To the best of our knowledge, so far the inexact resolvent inclusion problem \beqref{eq:PerturbedResolventProblem} has neither been discussed in a thorough way nor  in a general setting. However, there is, in a few places, a closely related discussion on closely related versions of \beqref{eq:PerturbedResolventProblem}. This discussion is brief, not always direct and sometimes also scattered. The first related discussion is implicit in Rockafellar \cite[Proof of Proposition 3, p. 882]{Rockafellar1976jour} in which one can find an explicit formula concerning  the solution to the problem when the setting is the classical one, that is, $X$ is a real Hilbert space, $A$ is maximally monotone and $f:=\frac{1}{2}\|\cdot\|^2$. This result is sometimes briefly mentioned elsewhere, for instance in \cite[p. 420]{AhmadiKhatibzadeh2014jour},\cite[p. 331]{BrezisLions1978jour}, and \cite[p. 412]{Djafari-RouhaniKhatibzadeh2008jour}. 

The second related discussion is scattered in the paper of Auslender, Teboulle and Ben-Tiba \cite[Propositions 1, 2]{AuslenderTeboulleBen-Tiba1999jour}. They consider a finite-dimensional space and impose several assumptions on $f$. Existence and sometimes uniqueness have been shown, but no explicit formula for the solution was presented. See Remark \bref{rem:AuslenderTeboulleBen-Tiba}  below for more details regarding \cite{AuslenderTeboulleBen-Tiba1999jour}. The third and fourth relevant  places are in G{\'a}rciga Otero and Iusem  \cite[Proposition 3.3]{Garciga-OteroIusem2004jour}, and Iusem and G{\'a}rciga Otero  \cite[Proposition 7]{IusemOtero2001jour}, respectively, and the fifth place is in Burachik and Iusem \cite[Proposition 6.6.3, p. 236]{BurachikIusem2008book}. In all of these cases $X$ is a real reflexive  Banach space, $A$ is maximally monotone, single-valued and continuous, $f$ is assumed to be a Bregman function satisfying additional properties, a specific sequence $(x_n)_{n=0}^{\infty}$ is considered and this sequence is based on a certain implicit version of \beqref{eq:PerturbedResolventProblem}. It is shown that when $x_n$ is not a zero of $A$, then any point in a neighborhood of an exact  solution to \beqref{eq:PerturbedResolventProblem} solves the considered implicit version of  \beqref{eq:PerturbedResolventProblem}.

In the exact resolvent (namely, when $\eta=0$) the solution to \beqref{eq:PerturbedResolventProblem} is well known in the classical case where $X$ is a Hilbert space, $f=\frac{1}{2}\|\cdot\|^2$, and $A$ is maximally monotone: in this case $y=(I+\lambda A)^{-1}(x)$ (see, for example, \cite[p. 878]{Rockafellar1976jour}; this result is frequently attributed to Minty \cite{Minty1962jour}, but in that paper Minty \cite[p. 344]{Minty1962jour} proved it under the additional assumptions that $A$ is single-valued and continuous). The solution is known (although not very well known) also in settings which are more general than real Hilbert spaces and $f=\frac{1}{2}\|\cdot\|^2$, but it is somewhat  scattered both in the literature and in the manner in which it is formulated. See, for instance, \cite[Corollary  3.1]{BurachikScheimberg2000jour} and \cite[p. 477]{ReichSabach2009jour}.  
\end{remark}

\section{A strongly implicit version of the inexact resolvent inclusion problem}\label{sec:Implicit}
In various papers, among them  \cite[Algorithm 2.1]{BurachikScheimbergSvaiter2001jour}, \cite[the algorithms in Section 4]{OteroIusem2013jour}, \cite[Algorithm 1]{OteroSvaiter2004jour}, \cite[Algorithms I, II]{IusemOtero2001jour}, \cite[Method 1]{IusemPennanenSvaiter2003jour}, \cite[Algorithm 3.1]{ParenteLotitoSolodov2008jour},\cite[Algorithm (B)]{Rockafellar1976jour}, \cite[Algorithm 3.1]{SolodovSvaiter1999-2jour},\cite[Algorithm 1]{SolodovSvaiter1999-1jour},\cite[Relation (9)]{SolodovSvaiter2000incol},\cite[Algorithm 2.1]{SolodovSvaiter2001jour}, one can find versions of the resolvent inclusion problem \beqref{eq:PrincipleSystem} in which the error term $\eta$ is not arbitrary but instead should satisfy a condition which is related to the sought solution $(y,\xi)$ of \beqref{eq:PrincipleSystem}. More precisely, given a real reflexive Banach space $(X,\|\cdot\|)$, a fully Legendre function $f:X\to\R$, a maximally monotone operator $A:X\to2^{X^*}$, a point $x\in X$, a positive number $\lambda$ and certain real-valued functions $\Phi(\cdot,\cdot,\cdot,\cdot)$  and $\Psi(\cdot,\cdot,\cdot,\cdot)$, we seek a triplet $(\eta,y,\xi)\in X^*\times X\times X^*$ such that the following system of conditions is satisfied:  
\begin{subequations}\label{eq:StronglyImplicitResolventInclusion}
\begin{align}
&\xi \in A(y),\label{eq:StronglyImplicit xiA(y)}\\
&\eta=\xi+\frac{1}{\lambda}\left(\nabla f(y)-\nabla f(x)\right),\label{eq:StronglyImplicit eta=xi_y}\\
&\Phi(\eta,\xi,x,y)<\Psi(\eta,\xi,x,y).\label{eq:Phi<Psi}
\end{align}
\end{subequations}
In other words, the original system of conditions \beqref{eq:PrincipleSystem} becomes strongly implicit. Below we formulate a simple but general proposition which extends many of the strongly implicit versions of the resolvent inclusion problem in the literature of which we are  aware. Later (Sections  \bref{sec:SolodovSvaiter1999-1}-\bref{sec:ManyMore}) we apply this proposition to deduce the well-definedness of the algorithmic schemes mentioned above. Due to the strong implicit nature expressed in Proposition \bref{prop:ImplicitInexactness} below, it is not surprising that the result has a certain local character. This, in some sense, is similar to the case of the classical implicit function theorem. 

\begin{prop}\label{prop:ImplicitInexactness}
Let $(X,\|\cdot\|)$ be a real reflexive Banach space. Let $f:X\to\R$ be fully Legendre and $A:X\to 2^{X^*}$ be maximally monotone. Let $U\subseteq X^*$ be an open subset containing $0$ and let $\Phi:U\times X^*\times X^2\to\R$ and $\Psi:U\times X^*\times X^2\to\R$ be two functions. For all $x\in X$, all $\lambda\in (0,\infty)$ and all $\eta\in U$, denote 
\begin{subequations}\label{eq:y xi phi psi}
\begin{align}
\wt{y}(\eta)&:=(\nabla f+ \lambda A)^{-1}\left(\lambda\eta+\nabla f(x)\right),\\ 
\wt{\xi}(\eta)&:=\eta-\frac{1}{\lambda}(\nabla f(\wt{y}(\eta))-\nabla f(x)),\\ 
\phi(\eta)&:=\Phi(\eta,\wt{\xi}(\eta),x,\wt{y}(\eta)),\\
\psi(\eta)&:=\Psi(\eta,\wt{\xi}(\eta),x,\wt{y}(\eta)),\\
\theta(\eta)&:=\psi(\eta)-\phi(\eta). 
\end{align}
\end{subequations}
Assume that $\theta$ is lower semicontinuous at 0 (in particular, this occurs when $\psi$ is lower semicontinuous at 0 and $\phi$ is upper semicontinous at 0; this latter case occurs, in particular, when both functions are continuous at $0$) and also that $\theta(0)>0$ (in particular, this happens when $\phi(0)=0$ and $\psi(0)>0$). Then there is $r>0$ such that each $\eta\in X^*$ satisfying $\|\eta\|<r$ belongs to $U$ and for every such $\eta$ there exists a unique pair $(y,\xi)\in X\times X^*$ such that $(\eta,y,\xi)$ satisfies \beqref{eq:StronglyImplicitResolventInclusion}. Moreover, \beqref{eq:PrincipleExplicit} holds, namely $y=\wt{y}(\eta)$ and $\xi=\wt{\xi}(\eta)$ for all such $\eta$.
\end{prop}
\begin{proof}
Since $U$ is open and $0\in U$, because $\theta$ is lower semicontinuous  at 0, and because $\theta(0)>0$, for $\epsilon:=0.5\theta(0)$ there is $r>0$ small enough such that any $\eta\in X^*$ satisfying $\|\eta\|<r$ belongs to $U$ and we have $\theta(\eta)>\theta(0)-\epsilon=0.5\theta(0)>0$. Since $\theta=\psi-\phi$ we have $\phi(\eta)<\psi(\eta)$ for all such $\eta$. This inequality and \beqref{eq:y xi phi psi} imply that \beqref{eq:Phi<Psi} holds with $y:=\wt{y}(\eta)$ and $\xi:=\wt{\xi}(\eta)$. In addition, Proposition \bref{prop:PrincipleResolvent} implies that this pair $(y,\xi)$ is the unique pair in $X\times X^*$ which satisies \beqref{eq:PrincipleSystem} (that is, it satisfies \beqref{eq:StronglyImplicit xiA(y)}-\beqref{eq:StronglyImplicit eta=xi_y}).  
\end{proof}

A sufficient condition for $\phi$ and $\psi$ from \beqref{eq:y xi phi psi} to be continuous at 0 is that the functions $\Phi$, $\Psi$, $\nabla f$, $(\nabla f+\lambda A)^{-1}$ are continuous. Among these functions, the first three are often continuous. As shown in Section \bref{sec:Continuity}, there are various simple sufficient conditions which imply the continuity of the fourth one.

\section{Continuous dependence of the solution of \beqref{eq:PrincipleSystem} on some involved parameters and a continuity property of the protoresolvent}\label{sec:Continuity}

A well-known phenomenon which occurs frequently (but  not always) in the theory of differential equations is the phenomenon of well-posed problems \cite[pp. 141--142]{Walter1998book}  (problems having this property are sometimes also called ``properly posed'' \cite[p. 227]{CourantHilbert1962IIbook}). The meaning of this notion is that there exists a unique solution to the considered  problem and this solution depends continuously on key parameters which describe the problem, that is, small perturbations in these parameters cause the solution of the problem to change only slightly (this continuous dependence  phenomenon is also called ``stability'' \cite[p. 2]{PinchoverRubinstein2005book}). We already know from Proposition \bref{prop:PrincipleResolvent} that  \beqref{eq:PrincipleSystem} has a unique solution. As is shown in Proposition \bref{prop:ContinuousDependence} below, if the both $\nabla f$ and the protoresolvent are  continuous, then the continuous dependence phenomenon occurs also in the case of \beqref{eq:PrincipleSystem}. As a result, frequently the inexact resolvent inclusion  problem \beqref{eq:PrincipleSystem} is well posed.

In what follows we first formulate Proposition \bref{prop:ContinuousDependence}. Then we formulate  several simple sufficient conditions which guarantee the continuity of the protoresolvent  (Corollary \bref{cor:ResolventContinuous}, Examples  \bref{ex:FiniteDimension}--\bref{ex:PowerNorm}) and also introduce (Definition \bref{def:ExtendStronglyMonotone} below) a certain generalization of the notion of strong monotonicity (our generalization is a variation of \cite[Definition 22.1, p. 383]{BauschkeCombettes2017book}). The usefulness of the assertions discussed here will become clear in Sections \bref{sec:SolodovSvaiter1999-1}--\bref{sec:ManyMore} below when we use them, together with the result about the strongly implicit version of the inexact resolvent inclusion problem (Proposition \bref{prop:ImplicitInexactness}), to prove the well-definedness of various inexact algorithmic schemes.  

\begin{prop}\label{prop:ContinuousDependence}
Under the conditions of Proposition \bref{prop:PrincipleResolvent}, suppose that both $\nabla f$ and   $(\nabla f+\lambda A)^{-1}$ are continuous. For each pair $(x,\eta)\in X\times X^*$, denote by  $(\wt{y}(x,\eta),\wt{\xi}(x,\eta))$ the unique solution in $X\times X^*$ to \beqref{eq:PrincipleSystem}. Then $\wt{y}(\cdot,\cdot)$ and $\wt{\xi}(\cdot,\cdot)$ are continuous functions. 
\end{prop}
\begin{proof}
This assertion follows immediately from Proposition \bref{prop:PrincipleResolvent} and \beqref{eq:PrincipleExplicit}.
\end{proof}
Additional types of a ``well-behaved'' dependence of the solutions of optimization problems on some of the involved parameters can be found in \cite{BonnansShapiro2000book}. Now we continue with a definition and a lemma. 

\begin{defin}\label{def:ExtendStronglyMonotone}
Let $(X,\|\cdot\|)$ be a real normed space. An operator $B:X\to 2^{X^*}$ is called uniformly monotone with modulus $\mu$ and pre-modulus $\wt{\mu}$ if 
\begin{equation}\label{eq:UniformMonotoneMu}
\langle u_1-u_2, y_1-y_2\rangle\geq \mu(\|y_1-y_2\|),\quad \forall\, y_1,y_2\in X, u_1\in By_1, u_2\in By_2,
\end{equation}
where $\mu:[0,\infty)\to[0,\infty)$ has the form $\mu(t)=t\wt{\mu}(t)$ for all $t\in [0,\infty)$, and where $\wt{\mu}:[0,\infty)\to [0,\infty)$ is increasing and invertible. $B$ is called uniformly monotone of power type $\rho$ if there are $\beta>0$ and $\rho>1$ such that $B$ is uniformly monotone with modulus $\mu(t):=\beta t^{\rho}$ for all $t\in [0,\infty)$. In other words, 
\begin{equation}\label{eq:Bhp-StronglyMonotone}
\langle u_1-u_2, y_1-y_2\rangle\geq \beta\|y_1-y_2\|^{\rho},\quad \forall\, y_1,y_2\in X, u_1\in By_1, u_2\in By_2, 
\end{equation}
A uniformly monotone operator of power type 2 is called strongly monotone. 
\end{defin}
A useful property of a pre-modulus $\wt{\mu}$ is that 
\begin{equation}\label{eq:mu(0)=0}
\wt{\mu}(0)=0=\wt{\mu}^{-1}(0).
\end{equation}
Indeed, $\wt{\mu}(0)\in [0,\infty)$ by our assumption. If, to the contrary, $\wt{\mu}(0)>0$, then $\wt{\mu}(t)\geq \wt{\mu}(0)>0$ for all $t\in [0,\infty)$ since $\wt{\mu}$ is increasing. Hence no $t\in [0,\infty)$ satisfies $\wt{\mu}(t)\in [0,\wt{\mu}(0))$, a contradiction to the assumption that $\wt{\mu}$ is onto $[0,\infty)$.

\begin{lem}\label{lem:ResolventContinuous}
Let $(X,\|\cdot\|)$ be a real normed space. Assume that $A:X\to 2^{X^*}$ is monotone and $B:X\to 2^{X^*}$ is uniformly monotone with pre-modulus $\wt{\mu}$. Then for all $w_1,w_2\in X^*$, $x_1\in (A+B)^{-1}w_1$, and $x_2\in (A+B)^{-1}w_2$, one has
\begin{equation}\label{eq:tilde mu}
\|x_1-x_2\|\leq \widetilde{\mu}^{-1}(\|w_1-w_2\|).
\end{equation}
In particular, if $B$ is uniformly monotone of power type $\rho>1$, then for all $w_1,w_2\in X^*$, $x_1\in (A+B)^{-1}w_1$, $x_2\in (A+B)^{-1}w_2$, one has
\begin{equation}\label{eq:hHolderContinuity}
\|x_1-x_2\|\leq \left(\frac{\|w_1-w_2\|}{\beta}\right)^{\frac{1}{\rho-1}}. 
\end{equation}
\end{lem}

\begin{proof}
Let $w_1,w_2\in X^*$. The assertion is trivial (void) if either $(A+B)^{-1}w_1=\emptyset$ or   $(A+B)^{-1}w_2=\emptyset$. Hence from now on we assume that $(A+B)^{-1}w_1\neq\emptyset$, $(A+B)^{-1}w_2\neq\emptyset$. Let $x_1\in (A+B)^{-1}w_1$ and $x_2\in (A+B)^{-1}w_2$. If $x_1=x_2$, then \beqref{eq:tilde mu} holds by \beqref{eq:mu(0)=0}. Assume from now on that $x_1\neq x_2$. By the definition of the inverse operator we have $w_1\in (A+B)x_1$ and $w_2\in (A+B)x_2$. Thus $(A+B)x_1$ and $(A+B)x_2$ are nonempty. Since  $(A+B)x_1=Ax_1+Bx_1$ and since, by definition, a sum of two sets is nonempty if and only if both sets are nonempty, we have  $Ax_1\neq \emptyset$, $Bx_1\neq \emptyset$, $Ax_2\neq \emptyset$, $Bx_2\neq \emptyset$. Let 
$a_1\in Ax_1$, $b_1\in Bx_1$, $a_2\in Ax_2$, $b_2\in Bx_2$ satisfy $w_1=a_1+b_1$, $w_2=a_2+b_2$. From these equalities, the monotonicity of $A$ and \beqref{eq:UniformMonotoneMu}, we have 
\begin{multline}\label{eq:h(x,x')>=beta|w-w'|}
\langle w_1-w_2,x_1-x_2\rangle=\langle a_1+b_1-a_2-b_2,x_1-x_2\rangle\\
=\langle a_1-a_2,x_1-x_2\rangle+ \langle b_1-b_2,x_1-x_2\rangle\geq 0+ \mu(\|x_1-x_2\|).
\end{multline}
Since $\|w_1-w_2\|\|x_1-x_2\|\geq \langle w_1-w_2,x_1-x_2\rangle$, by the definition of the norm in $X^*$, it follows from \beqref{eq:h(x,x')>=beta|w-w'|} that $\|w_1-w_2\|\|x_1-x_2\|\geq \mu(\|x_1-x_2\|)$. Since $x_1\neq x_2$ and $\mu(t)=t\wt{\mu}(t)$ for all $t\geq 0$, the fact that $\widetilde{\mu}^{-1}$ exists and is increasing implies \beqref{eq:tilde mu}. Finally, when $\mu$ is  uniformly monotone of power type $\rho>1$, then $\widetilde{\mu}(t)=\beta t^{\rho-1}$ and $\widetilde{\mu}^{-1}(t)=(t/\beta)^{1/(\rho-1)}$ for each $t\in [0,\infty)$. Hence \beqref{eq:tilde mu} implies \beqref{eq:hHolderContinuity}. 
\end{proof}

\begin{cor}\label{cor:ResolventContinuous}
Under the assumptions of Lemma \bref{lem:ResolventContinuous}, if, in addition, $(A+B)^{-1}$ is single-valued, then $(A+B)^{-1}$ is continuous. In particular, $(\nabla f+\lambda A)^{-1}$ is H\"older continuous with an exponent $1/(\rho-1)$ under the following   slight strengthening of the assumptions of Lemma  \bref{lem:ResolventOrdinaryFunction}:  $(X,\|\cdot\|)$ is a real reflexive Banach space, $A$ is maximally monotone, $\lambda>0$, and $f:X\to\R$ is fully Legendre and has the property that $\nabla f$ is uniformly monotone of power type $\rho>1$; moreover, if, in addition, $\nabla f$ is continuous, then $\Res^f_{\lambda A}$ is continuous. 
\end{cor}

\begin{proof}
The first assertion follows from Lemma \bref{lem:ResolventContinuous} because  $(A+B)^{-1}$ is single-valued and $\wt{\mu}^{-1}$ is continuous (since it is one-dimensional, increasing and invertible) and satisfies \beqref{eq:mu(0)=0}. The second assertion follows from the first one by using Lemma  \bref{lem:ResolventOrdinaryFunction}, replacing $A$ with $\lambda A$, taking $B:=\nabla f$, and using \beqref{eq:hHolderContinuity}. The assertion regarding $\Res^f_{\lambda A}$ follows from the second assertion and \beqref{eq:Res^f_A}. 
\end{proof}

\begin{expl}\label{ex:FiniteDimension}
Suppose that the assumptions of  Proposition \bref{prop:PrincipleResolvent} hold where $X$ is $\R^m$ with the Euclidean norm (or any other norm), $m\in\N$. Lemma  \bref{lem:ResolventOrdinaryFunction} implies that $(\nabla f +\lambda A)^{-1}$ is maximally monotone and single-valued. Hence we can use \cite[Corollary 2, p.  166]{BurachikIusemSvaiter1997jour}  or \cite[Theorem 12.63(c), p. 568]{RockafellarWets1998book} to conclude that $(\nabla f +\lambda A)^{-1}$ is continuous. 
\end{expl}

\begin{expl}\label{ex:HilbertContinuous}
Suppose that the assumptions of  Proposition \bref{prop:PrincipleResolvent} hold in the case where  $(X,\|\cdot\|)$ is a Hilbert space and $f=\frac{1}{2}\|\cdot\|^2$. Thus $\nabla f=I$ and hence $\nabla f$ is strongly monotone (Definition \bref{def:ExtendStronglyMonotone}). Hence Corollary \bref{cor:ResolventContinuous} generalizes the well-known fact that in a Hilbert space the operator $(I+\lambda A)^{-1}$ is  nonexpansive \cite[Proposition  1(c)]{Rockafellar1976jour}, \cite[Corollary 23.9, p. 396]{BauschkeCombettes2017book}. More generally, when $f(x):=\frac{1}{2}\langle Bx,x\rangle$, $x\in X$, where $B:X\to X$ is a symmetric, continuous, invertible and strongly monotone linear operator, then $f$ is fully Legendre (a conclusion which follows from Example \bref{ex:PositiveDefinite} because $B$ must be positive definite) and $B=\nabla f$. Therefore from Corollary \bref{cor:ResolventContinuous} we conclude that  $(B+\lambda A)^{-1}$ is Lipschitz continuous. In particular, $(B+\lambda A)^{-1}$ is Lipschitz continuous if $X$ is finite dimensional and $B$ is a positive definite (thus symmetric) linear operator because then $\langle Bx,x\rangle\geq \beta\|x\|^2$, where $\beta:=\inf\{\langle Bx,x\rangle: x\in X,\|x\|=1\}$ [because in this case   $\beta=\langle Bx_0,x_0\rangle$ for some $x_0$ belonging to the unit sphere by the compactness of the sphere (since the space is finite-dimensional), so the fact that $B$ is positive definite implies that $\langle Bx_0,x_0\rangle>0$; in addition, the finite dimensionality of the space implies that $B$ is continuous; since $B$ is positive definite, it is one-to-one and hence the finite dimensionality implies that $B$ is also invertible].  
\end{expl}

\begin{expl}\label{ex:PowerNorm}
Suppose that the assumptions of  Proposition \bref{prop:PrincipleResolvent} hold where  $(X,\|\cdot\|)$ is a Banach space which is smooth and has a modulus of convexity of power  type $\rho>1$, namely, there exists $\beta>0$ and $\rho>1$ such that  modulus of convexity $\delta_X(\epsilon):=\inf\{1-0.5\|x_1+x_2\|: x_1,x_2\in X, \|x_1\|=\|x_2\|=1, \|x_1-x_2\|\geq\epsilon\}$,  $\epsilon\in [0,2]$, satisfies $\delta_X(\epsilon)\geq \beta\epsilon^{\rho}$ for all $\epsilon\in [0,2]$. Well-known examples of spaces having this property are the $L_p[0,1]$ and $\ell_p$ spaces, $p\in (1,\infty)$, where in this case $\rho=\max\{2,p\}$ (details about the power type $\rho$ property can be found in  \cite[pp. 63, 81]{LindenstraussTzafriri1979book} and \cite[p. 69]{Diestel1975book}; smoothness is, of course, just a consequence of the well-known facts that $X^*$ is isometric to, respectively,  $L_q[0,1]$ or $\ell_q$ for $q\in (1,\infty)$ satisfying $(1/p)+(1/q)=1$, and the fact that a Banach space is uniformly convex if and only if its dual is uniformly smooth \cite[Theorem 3.7.9, p. 236]{Zalinescu2002book}). As follows from \cite[p. 258]{McCarthy1967jour} or \cite[Theorem 2.2]{Tomczak-Jaegermann1974jour}, the moduli of convexity of the $c_p$ spaces, $p\in (1,\infty)$ (denoted by $S_p$ in  \cite{Tomczak-Jaegermann1974jour}) are also of power type  $\rho=\max\{2,p\}$. 

Since $\rho>0$ and $\delta_X$ is of power type $\rho$, it follows that $\delta_X(\epsilon)>0$ whenever $\epsilon>0$, and so $X$ is uniformly convex and hence reflexive \cite[pp. 76-78]{Brezis2011book}. Now let $f:X\to\R$ be defined by $f(x):=(1/\rho)\|x\|^{\rho}$ for each $x\in X$. Then $f$ is fully Legendre (Example \bref{ex:PowerRho} above). In addition, both $f$ and $f^*$ are smooth \cite[Lemma 6.2]{BauschkeBorweinCombettes2001jour}. Since, as is well known \cite[Theorem 5.37, p. 77]{VanTiel1984book}, the subgradient of a G\^ateaux differentiable convex function coincides with  the singleton containing the gradient of the function, it follows that $B:=\partial f=\{\nabla f\}$, and by the usual abuse of notation $B=\nabla f$. On the other hand, according to \cite[p. 194]{XuRoach1991jour} and the fact that $X$ is reflexive, we have $\partial f=J_{\rho}$ where $J_{\rho}$ is the duality mapping with gauge function $t\mapsto t^{\rho-1}$, $t\in[0,\infty)$. By \cite[p. 194 and Theorem 1(ii), p. 195]{XuRoach1991jour}, there exists $\alpha>0$ such that for all $y_1,y_2\in X$
\begin{equation}\label{eq:delta_X}
\langle By_1-By_2,y_1-y_2\rangle\geq\alpha\left(\max\{\|y_1\|,\|y_2\|\}\right)^{\displaystyle{\rho}}
\delta_X\left(\frac{\|y_1-y_2\|}{2\max\{\|y_1\|,\|y_2\|\}}\right). 
\end{equation}
(Here and in \cite{XuRoach1991jour} one should assume that $\|y_1\|\neq 0$ or $\|y_2\|\neq 0$; when $\|y_1\|=0=\|y_2\|$, then we define the right-hand side to be 0 so that \beqref{eq:delta_X} is satisfied in this case too.) Since we assume that $\delta_X$ is of power type $\rho$, it follows from \beqref{eq:delta_X}  that there exists $\beta>0$ such that 
\begin{equation}\label{eq:delta_Xp_type}
\langle By_1-By_2,y_1-y_2\rangle\geq \frac{\alpha\beta}{2^{\rho}}\|y_1-y_2\|^{\rho},\quad 
 \forall y_1,y_2\in X.
\end{equation}
Therefore Corollary \bref{cor:ResolventContinuous} ensures that $(\nabla f+\lambda A)^{-1}$ is H\"older continuous with exponent $1/(\rho-1)$. We note that the particular case where $A\equiv 0$, $f=\frac{1}{2}\|\cdot\|^2$ and both $\delta_X$ and $\delta_{X^*}$ are of power type $2$, is implicit in \cite[Proof of Proposition 3.2]{BruckReich1977jour}. 
\end{expl}

\section{Well-definedness of Eckstein \cite[Algorithm (10) and Theorem 1]{Eckstein1998jour}}\label{sec:Eckstein}
The paper \cite{Eckstein1998jour} discusses an inexact version of the proximal point algorithm in $X=\R^m$  with  the Euclidean norm, where $m\in\N$ is fixed and where the iterations are based on a general Bregman function $f$ and a maximally monotone operator $A$. The goal of this  algorithmic scheme is to find a zero of $A$ (as a matter of fact, the setting in \cite{Eckstein1998jour} is a bit different, but  it coincides with the one discussed here because we consider Bregman functions the effective domain of which is $X$; see Remark \bref{rem:EcksteinSetting} below). This scheme is defined as follows: 
\begin{equation}
x_0\in X\,\,\textnormal{is arbitrary},
\end{equation}
\begin{equation}\label{eq:Eckstein}
\eta_{n+1}+\nabla f(x_n)\in \lambda_nA(x_{n+1})+\nabla f(x_{n+1}),\quad \forall n\in \N\cup\{0\}.
\end{equation}
Here $(\lambda_n)_{n=0}^{\infty}$ is a sequence of positive numbers and $(\eta_n)_{n=1}^{\infty}$ are arbitrary vectors in $X$ which are regarded as being the error terms (as we have already observed, frequently these error terms are unknown to the users: for instance, they may appear during the iterative process due to computational errors, one may be able to evaluate only their magnitude, and so on). The function $f$ is  assumed to be a Bregman function. In the context of \cite{Eckstein1998jour} this means that $f$ satisfies \cite[Conditions B1-B7]{Eckstein1998jour}. Thus $f$ is  strictly convex and continuously differentiable in $X$ and $L(x,\alpha):=\{y\in X: D_f(x,y)\leq\alpha\}$ is bounded for all $x\in X$ and $\alpha\in\R$, where $D_f$ is the Bregman distance associated with $f$ as defined in \beqref{eq:D_f}. 
Under these conditions, the assumption that $(x_n)_{n=1}^{\infty}$ is well defined, the assumption that $A^{-1}(0)\neq\emptyset$, and the assumptions that $\sum_{n=1}^{\infty}\|\eta_n\|<\infty$ and $\sum_{n=1}^{\infty}\langle \eta_n,x_n\rangle$ exists and is finite,  it is shown in \cite[Theorem 1]{Eckstein1998jour} that $(x_n)_{n=0}^{\infty}$ converges to a zero of $A$. 

In the text that precedes the formulation of \cite[Theorem 1]{Eckstein1998jour}, namely in \cite[the beginning of Section 3]{Eckstein1998jour}, there is a limited discussion regarding the issue of existence of a sequence $(x_n)_{n=0}^{\infty}$ satisfying \beqref{eq:Eckstein}. Indeed, only in the case where all the error terms are equal to zero a sufficient condition was presented to ensure the existence of  $(x_n)_{n=0}^{\infty}$ (in our context, since we assume that $f$ is defined on $X$ and is finite there, this condition reduces to the assumption that $\nabla f$ maps $X$ onto $X$). The case where one or more of the error terms are not equal to zero has not been considered. In Theorem \bref{thm:Eckstein} below we show that when $f$ is fully Legendre (an assumption which implies, according to Remark \bref{rem:FullyLegendreFIniteDim} below, that $f$ is a Bregman function) then Eckstein's algorithm is well defined for arbitrary initial points and arbitrary error terms. 
\begin{thm}\label{thm:Eckstein}
Assume that $f:X\to\R$ is fully Legendre. Then for each $x_0\in X$, each sequence $(\lambda_n)_{n=0}^{\infty}$ of positive numbers and each sequence $(\eta_n)_{n=1}^{\infty}$ of vectors in $X$ there exists a unique sequence $(x_n)_{n=1}^{\infty}$ of elements in $X$ such that \beqref{eq:Eckstein} holds. Moreover, $x_{n+1}=(\nabla f+\lambda_{n}A)^{-1}(\eta_{n+1}+\nabla f(x_n))$ for all $n\in\N\cup \{0\}$.
\end{thm}
\begin{proof}
A simple verification shows that  \beqref{eq:Eckstein} holds if and only if for each $n\in \N\cup\{0\}$ we have $\eta_{n+1}/\lambda_{n}\in Ax_{n+1}+(1/\lambda_{n})(\nabla f(x_{n+1})-\nabla f(x_n))$. Proposition \bref{prop:PrincipleResolvent} and induction imply that for all $n\in\N\cup\{0\}$ there exists a unique $x_{n+1}\in X$ which satisfies  \beqref{eq:Eckstein},  and, actually $x_{n+1}=(\nabla f+\lambda_n A)^{-1}(\lambda_{n}(\eta_{n+1}/\lambda_{n})+\nabla f(x_n))$. 
\end{proof}

\begin{remark}\label{rem:EcksteinSetting}
As we briefly mentioned above, the setting in \cite{Eckstein1998jour} is a bit different from what we assumed above, in fact it is more general,  because $f$ is defined on $X$ but it may attain any value in $(-\infty,\infty]$ in such a way that its effective domain is a convex subset of $X$ with a  nonempty interior $S$; in addition, \cite[Conditions  B1--B7]{Eckstein1998jour} should hold for points in $S$ or in the closure of $\dom(f)$; additional assumptions on $A$ are imposed, namely one condition from  \cite[Conditions  A1-A3]{Eckstein1998jour} and also the condition that the intersection of the relative interior of $\dom(A)$ and $S$ is nonempty. When all of these conditions hold, together with the ones mentioned after \beqref{eq:Eckstein}, and, in addition,  it is assumed that $\hat{A}:=A+N_{\overline{\dom(f)}}$ has a zero, then \cite[Theorem 1]{Eckstein1998jour} implies that $(x_n)_{n=0}^{\infty}$ converges to a zero of $\hat{A}$. Here $N_{\cl{\dom(f)}}$ is the normal cone operator with respect to $\cl{\dom(f)}$, where, for all nonempty subset $C$ of $X$ and all $x\in X$, if $x\in C$, then  $N_C(x):=\{z\in X: \langle z,w-x\rangle\leq 0,\,\forall w\in C\}$ and if $x\notin C$, then $N_C(x):=\emptyset$. In our case $C=X$, and a simple verification shows that $N_X(x)=0$ for all $x\in X$. Thus $\hat{A}=A$ and \cite[Theorem 1]{Eckstein1998jour}  implies that $(x_n)_{n=0}^{\infty}$ converges to a zero of $A$.  
\end{remark}

\section{Well-definedness of Reich-Sabach \cite[Algorithm (4.1) and Theorem 4.1]{ReichSabach2010b-jour} }
\label{sec:ReichSabach2010b4.1} 
The paper \cite{ReichSabach2010b-jour} introduces in  \cite[Algorithm  (4.1)]{ReichSabach2010b-jour} an algorithmic scheme the goal of which is to find a common zero of a finite family of maximally monotone operators in an arbitrary real reflexive Banach space. Following the notation of  \cite{ReichSabach2010b-jour}, we now present their scheme: 
\begin{subequations}\label{eq:4.1}
\begin{align}
x_0&\in X,\\
\eta^i_n&=\xi^i_n+\displaystyle{\frac{1}{\lambda^i_n}}\left(\nabla f(y^i_n)-\nabla f(x_n)\right),\quad \xi^i_n\in A_i(y^i_n),\\
w^i_n&=\nabla f^*\left(\lambda^i_n\eta^i_n+\nabla f(x_n)\right),\\
C^i_n&=\{z\in X: D_f(z,y^i_n)\leq D_f(z,w^i_n)\},\\
C_n:&=\cap_{i=1}^N C^i_n,\\
Q_n&=\{z\in X: \langle\nabla f(x_0)-\nabla f(x_n),z-x_n\rangle \leq 0\},\\
\,x_{n+1}&=\textnormal{proj}^f_{C_n\cap Q_n}(x_0), \quad n=0,1,2,\ldots.
\end{align}
\end{subequations}
Here $n\geq 0$ is an integer, $N$ is a fixed natural number, $i$ is a natural number in $\{1,\ldots, N\}$, and  $\lambda^i_n>0$ for each such $n$ and $i$. For all such $i\in \{1,\ldots, N\}$, the operator  $A_i$ is a  maximally monotone operator from $X$ to $2^{X^*}$. 
We also assume that the common zero set of the operators is nonempty, that is,
\begin{equation}\label{eq:ZeroSet}
Z:=\bigcap_{i=1}^N A_i^{-1}(0)\neq\emptyset. 
\end{equation}
The function $f:X\to\R$ has the property that on each nonempty convex and bounded subset of $X$  it ($f$) is assumed to be bounded, uniformly Fr\'echet differentiable, and totally convex on bounded subsets of $X$, where total convexity at a point $x\in X$ means that 
\begin{equation}\label{eq:TotalConvexity}
\nu_f(x,t):=\inf\{D_f(y,x): y\in X, \|y-x\|=t\}>0,\quad\,\forall x\in X,\forall t>0,
\end{equation}
and total convexity on bounded subsets of $X$ (called sequential consistency in \cite[p.  65]{ButnariuIusem2000book}) means that for each bounded subset $E\subset X$, 
\begin{equation}\label{eq:TotalConvexityBounded}
\nu_f(E,t):=\inf\{\nu_f(x,t): x\in E\}>0,\quad\forall t>0.
\end{equation}
A totally convex function must be strictly convex as follows from \cite[Proposition 1.2.6(i), p. 27]{ButnariuIusem2000book}. As a matter of fact, in our case $f$  is even uniformly convex at each $x\in X$ because it is totally convex and Fr\'echet  differentiable \cite[Proposition 2.3, p. 38]{ButnariuIusemZalinescu2003jour}, but we do not need this stronger result. Here $D_f$ is the Bregman distance  (Bregman divergence) associated with the Bregman function $f$, that is,
\begin{equation}\label{eq:D_f}
D_f(y,x):=f(y)-f(x)-\langle\nabla f(x),y-x\rangle,\quad \forall\, x,y\in X.
\end{equation}
The convex conjugate $f^*$ is assumed to be bounded (hence finite) and uniformly Fr\'echet differentiable on bounded subsets of $X^*$. The expression $\textnormal{proj}^f_C(x)$ represents the (right) Bregman projection of $x\in X$ onto a nonempty, closed and convex subset $C$ of $X$ defined by  $\textnormal{proj}^f_C(x):=\textnormal{argmin}\{D_f(y,x): y\in C\}$.  This operator is well defined, that is, there exists a unique $y(x,C)\in C$ such that $D_f(y(x,C),x)=\inf\{D_f(y,x): y\in C)\}$; see \cite[Corollary 7.9]{BauschkeBorweinCombettes2001jour}.  As explained in Remark \bref{rem:ReichSabachFullyLegendre}  below, the above assumptions imply that $f$ must be fully Legendre (in particular, there is no need to assume in advance that $f$ is Legendre). 

 Under the assumption that \beqref{eq:4.1} is well defined and that  $\liminf_{n\to\infty}\lambda^i_n>0$ and $\lim_{n\to\infty}\eta^i_n=0$ for each $i\in\{1,\ldots,N\}$, it was shown in the proof of \cite[Theorem 4.1]{ReichSabach2010b-jour} that $(x_n)_{n=0}^{\infty}$  converges strongly to  $\textnormal{proj}^f_Z(x_0)$. The sequences  $(\eta^i_n)_{n=0}^{\infty}$, $i\in \{1,\ldots,N\}$ were regarded as the error terms and in the formulation and proof of \cite[Theorem 4.1]{ReichSabach2010b-jour} it was shown that \beqref{eq:4.1} is well defined only under the assumption that all of these error terms are equal to zero. In fact,  in \cite[Algorithm (4.1)]{ReichSabach2010b-jour} and in \beqref{eq:4.1} there is a slight ambiguity regarding some parameters (for instance, whether the $\eta^i_n$ can be arbitrary or perhaps they should be defined in terms of other parameters).
 
In Theorem \bref{thm:RS2010bTheorem4.1} below we show that \beqref{eq:4.1} is well defined for arbitrary $x_0\in X$ and arbitrary error terms $\eta^i_n\in X^*$, $i\in\{1,\ldots,N\}$, $n\in\N\cup\{0\}$. This theorem also clarifies how the various parameters presented in \beqref{eq:4.1} should be handled by the users (for example, $y^i_n$ and $\xi^i_n$ should satisfy \beqref{eq:4.1Improved-y} and \beqref{eq:4.1Improved-xi} below, respectively, for each $i\in \{1,\ldots,N\}$ and each nonnegative  integer $n$). Its proof is based on Proposition \bref{prop:PrincipleResolvent} and the simple but important observation that \beqref{eq:4.1} is actually a system of conditions (mainly equations and inclusions) on a tuple of unknowns, a system which may have one solution, multiple solutions, or may not have any solution at all. 
 
\begin{thm}\label{thm:RS2010bTheorem4.1}
Let $x_0\in X$ be arbitrary. Then for each $i\in \{1,\ldots,N\}$, each nonnegative integer $n$, each $\eta^i_n\in X^*$, and each $\lambda^i_n>0$ there exist unique  $w^i_n\in X$, $y^i_n\in X$, $\xi^i_n\in X^*$, $C^i_n\subseteq X$, $C_n\subseteq X$, $Q_n\subseteq X$, and $x_{n+1}\in X$ which satisfy \beqref{eq:4.1}. Moreover, these parameters satisfy the following system of conditions:
\begin{subequations}\label{eq:4.1Improved}
\begin{align}
w^i_n&=\nabla f^*\left(\lambda^i_n\eta^i_n+\nabla f(x_n)\right),\label{eq:4.1Improved-w}\\
y^i_n&=(\lambda^i_n A_i+\nabla f)^{-1}(\lambda^i_n\eta^i_n+\nabla f(x_n)),\label{eq:4.1Improved-y}\\
\xi^i_n&=\eta^i_n-\frac{1}{\lambda^i_n}(\nabla f(y^i_n)-\nabla f(x_n)),\label{eq:4.1Improved-xi}\\
\xi^i_n&\in A_i(y^i_n),\label{eq:4.1-xiInA_i}\\
C^i_n&=\{z\in X: D_f(z,y^i_n)\leq D_f(z,w^i_n)\},\label{eq:4.1Improved-C_i}\\
C_n&=\cap_{i=1}^N C^i_n,\label{eq:4.1Improved-C}\\
Q_n&=\{z\in X: \langle\nabla f(x_0)-\nabla f(x_n),z-x_n\rangle \leq 0\},\label{eq:4.1Improved-Q}\\
C_n&\cap Q_n\,\,\textnormal{is convex, closed, and contains}\,Z,\label{eq:4.1Improved-C_nQ_n}\\
\,\,\,\,\,\,x_{n+1}&=\textnormal{proj}^f_{C_n\cap Q_n}(x_0).\label{eq:4.1Improved-x}
\end{align}
\end{subequations}
\end{thm}

\begin{proof}
We apply induction on $n$. Let $n=0$ and fix $i\in\{1,\ldots,N\}$. The existence and uniqueness of the $w^i_0$ satisfying the third line of \beqref{eq:4.1} (and hence  \beqref{eq:4.1Improved-w}) is immediate because it is expressed in unique way in terms of $x_0$, $\eta^i_0$ and $\lambda^i_0$. Proposition \bref{prop:PrincipleResolvent} implies that there exists a unique pair $(y^i_0,\xi^i_0)\in X\times X^*$ such that the second line of \beqref{eq:4.1} holds, and moreover, according to this proposition, \beqref{eq:4.1Improved-y}--\beqref{eq:4.1-xiInA_i} hold. The existence and uniqueness of the $C^i_0$ satisfying the fourth line of \beqref{eq:4.1}  (and therefore satisfying \beqref{eq:4.1Improved-C_i}) is immediate since it is expressed in a unique way using $y^i_0$ and $w^i_0$ (which have just been derived). This is true for each $i\in\{1\ldots,N\}$. Since $x_0$ is known, the existence and uniqueness of the $C_0$ and $Q_0$ satisfying the fifth and sixth lines of \beqref{eq:4.1} respectively (and thus satisfying    \beqref{eq:4.1Improved-C} and \beqref{eq:4.1Improved-Q}, respectively) is immediate. 

It remains to prove that \beqref{eq:4.1Improved-C_nQ_n} holds and that there exists a unique  $x_1$ satisfying the seventh line of \beqref{eq:4.1}, that is, \beqref{eq:4.1Improved-x}. We first show that \beqref{eq:4.1Improved-C_nQ_n} holds. Once this is done, we can use the fact mentioned after \beqref{eq:D_f} that the Bregman projection of a point in $X$ on a nonempty, closed and convex subset of $X$ exists and is unique and hence $x_1$ is well-defined. The definition of $D_f$ in \beqref{eq:D_f} and the definition of $C^i_0$ imply that for each $i\in\{1\ldots,N\}$,
\begin{multline*}
C^i_0=\{z\in X: f(z)-f(y^i_0)-\langle \nabla f(y^i_0), z-y^i_0\rangle\leq 
f(z)-f(w^i_0)-\langle \nabla f(w^i_0),z-w^i_0\rangle\}\\
=\{z\in X: \langle \nabla f(w^i_0)-\nabla f(y^i_0),z\rangle\leq f(y^i_0)-f(w^i_0)-\langle \nabla f(y^i_0),y^i_0\rangle+\langle \nabla f(w^i_0),w^i_0\rangle\}.
\end{multline*}
Hence if $\nabla f(w^i_0)\neq \nabla f(y^i_0)$, then $C^i_0$ is a closed halfspace. Otherwise, either $C^i_0=X$ or $C^i_0=\emptyset$. We claim that the second possibility cannot be  satisfied. Indeed, since $y^i_0=\Res^f_{\lambda^i_0 A_i}(w^i_0)$ as follows from \beqref{eq:Res^f_A},  \beqref{eq:4.1Improved-w}, \beqref{eq:4.1Improved-y} and Lemma \bref{lem:f^*f^-1},  we can use  \cite[Proposition 2.8]{ReichSabach2010b-jour} from which it follows that $D_f(u,y^i_0)=D_f(u,\Res^f_{\lambda^i_0 A_i}(w^i_0))\leq D_f(u,w^i_0)$ holds for each $u$ in the common zero set $Z$ from \beqref{eq:ZeroSet}. This and \beqref{eq:4.1Improved-C_i}  imply that $u\in C^i_0$ for each $i\in\{1,\ldots,N\}$. Thus $C^i_0\neq \emptyset$ for every $i\in\{1\ldots,N\}$ and actually $Z\subseteq C_0=\cap_{i=1}^N C^i_0$. Finally, $Q_0=X$ and hence obviously $Z\subseteq Q_0$. It follows that 
$C_0\cap Q_0$ is an intersection of sets which are either closed halfspaces or the whole space and hence $C_0\cap Q_0$ is closed and convex and it contains $Z$. 

So far we have proved the assertion for the case $n=0$. Now we can increment $n$ and use induction on it by repeating the above reasoning in the induction step. The only difference is that $Q_n$ will usually be a halfspace and not the whole space, and so the inclusion $Z\subseteq C_n\cap Q_n$ is not immediate. However, we do have $Z\subseteq C_n\cap Q_n$. Indeed, in the induction step we can prove that $Z\subseteq C_n$ in a similar way to the proof that the inclusion $Z\subseteq C_0$ was proved in the previous paragraph. Since from the induction hypothesis $K:=C_{n-1}\cap Q_{n-1}$ is closed and convex and $Z\subseteq K$, we can use \cite[Proposition  2.6(i),(ii)]{ReichSabach2010b-jour} and the fact that $x_n=\textnormal{proj}^f_{K}(x_0)\in K$ to conclude that $x_n$ satisfies the variational  inequality $\langle \nabla f(x_0)-\nabla f(x_n),u-x_n\rangle \leq 0$ for all $u\in K$ and in particular for all $u\in Z$. We conclude from \beqref{eq:4.1Improved-Q} that  $Z\subseteq Q_n$.
 Consequently, $Z\subseteq C_n\cap Q_n$, as claimed. As a final remark we note that $C_n\cap Q_n$ is closed and convex because it is an intersection of nonempty sets which are either closed halfspaces or the whole space. 
\end{proof}

\begin{remark}
We take this opportunity to correct a misprint in \cite{ReichSabach2010b-jour}: 
the expression $H_n\cap W_n$ in \cite[Algorithm (4.4), p. 35]{ReichSabach2010b-jour} should be replaced by $C_n\cap Q_n$. 
\end{remark}

\section{Well-definedness of Solodov-Svaiter \cite[Algorithm 1.1 and Theorems 2.2, 2.4]{SolodovSvaiter1999-1jour}}\label{sec:SolodovSvaiter1999-1}

The paper \cite{SolodovSvaiter1999-1jour} discusses an inexact version of the proximal point algorithm in a Hilbert space $X$. The goal of the corresponding inexact algorithmic scheme \cite[Algorithm  1.1]{SolodovSvaiter1999-1jour} is to find a zero of a maximally monotone operator $A:X\to 2^X$ assuming that $A^{-1}(0)\neq\emptyset$. Here is the scheme:  

\begin{alg}\label{alg:SolodovSvaiter1999-1}
{\bf\noindent Initialization:} Choose an arbitrary $x_0\in X$, an arbitrary $\sigma\in [0,1)$,  and an arbitrary sequence of positive numbers $(\mu_n)_{n=0}^{\infty}$.\\

{\bf \noindent Iterative step:} Given $n\in\N\cup\{0\}$ and $x_n\in X$, find $(y_n,\xi_n,\eta_n)\in X^3$ satisfying the following conditions:
\begin{subequations}\label{eq:SolodovSvaiter1999}
\begin{align}
\xi_n&\in A(y_n)\label{eq:xi_n in A}\\
0&=\xi_n+\mu_n(y_n-x_n)+\eta_n,\label{eq:0=xi_n+}\\
\|\eta_n\|&\leq \sigma \max\{\|\xi_n\|,\mu_n\|y_n-x_n\|\}\label{eq:eta_n<= SolodovSvaiter1999}.
\end{align}
\end{subequations}
If $\xi_n=0$ or $y_n=x_n$, then stop. Otherwise let 
\begin{equation}\label{eq: x_n+1 SolodovSvaiter1999-1}
x_{n+1}:=x_n-\frac{\langle \xi_n,x_n-y_n\rangle}{\|\xi_n\|^2}\xi_n.
\end{equation}
\end{alg}

In order for this algorithm to be well defined, the existence of solutions $(y_n,\xi_n,\eta_n)$ to \beqref{eq:SolodovSvaiter1999} should be established. In \cite{SolodovSvaiter1999-1jour} only the case where $\sigma=0$ was discussed 
\cite[pp. 61-62]{SolodovSvaiter1999-1jour} and it was written that in this case the algorithmic scheme reduces to the exact case ($\eta_n=0$), namely, to the classical exact resolvent inclusion problem (\beqref{eq:PerturbedResolventProblem} in which $X$ is a real Hilbert space, $\eta=0$ and $f=\frac{1}{2}\|\cdot\|^2$) which is known to have a unique solution. In  fact, if one denotes $y_n:=y$, $\xi_n:=\xi$ and $\eta_n:=0$ where $(y,\xi)$ is the unique solution to \beqref{eq:PrincipleSystem} when $\eta=0$, then the triplet $(y_n,\xi_n,\eta_n)$ solves \beqref{eq:SolodovSvaiter1999} even if $\sigma>0$. However, it is not clear from  \cite{SolodovSvaiter1999-1jour} whether there exist solutions $(y_n,\xi_n,\eta_n)$ to \beqref{eq:SolodovSvaiter1999} such that $\eta_n\neq 0$, namely solutions which are to be expected in real-world scenarios. Anyway, under the assumption that $\sup\{\mu_n: n\in\N\}<\infty$ and that there exist sequences $(x_n)_{n=0}^{\infty}$ satisfying 
\beqref{eq:SolodovSvaiter1999}--\beqref{eq: x_n+1 SolodovSvaiter1999-1}, it was shown in \cite[Theorem 2.2]{SolodovSvaiter1999-1jour} that any such sequence converges weakly to a zero of $A$. Under further assumptions it was shown in \cite[Theorem 2.4]{SolodovSvaiter1999-1jour} that these sequences converge strongly to a zero  of $A$. 

The following theorem shows that Algorithm \bref{alg:SolodovSvaiter1999-1} is well defined even if $\sigma\geq 1$ (it is, however, an open problem whether the generated sequence converges to a zero of $A$, since the analysis in \cite{SolodovSvaiter1999-1jour} depends on the assumption that $\sigma\in [0,1)$). Moreover, if $\sigma>0$, then either $x_n$ is a zero of $A$ or \beqref{eq:SolodovSvaiter1999} has strictly inexact solutions. 
 
\begin{thm}\label{thm:SolodovSvaiter1999-1}
There exist sequences which satisfy Algorithm \bref{alg:SolodovSvaiter1999-1} in the exact and inexact cases, even if $\sigma\geq 1$. More precisely, for all $n\in \N\cup\{0\}$, if Algorithm \bref{alg:SolodovSvaiter1999-1} generates $x_n$ (namely, it does not terminate before iteration $n$), then at least one of the following possibilities holds:
\begin{enumerate}[(i)]
\item\label{item:x_n is zero of A} $x_n$ is a zero of $A$. In this case the triplet  $(y_n,\xi_n,\eta_n):=(x_n,0,0)$ satisfies  \beqref{eq:SolodovSvaiter1999}. 
\item $\sigma=0$. In this case there exists a unique triplet $(y_n,\xi_n,\eta_n)\in X^3$ such that \beqref{eq:SolodovSvaiter1999} holds, namely $((I+(1/\mu_n)A)^{-1}(x_n),-\mu_n(y_n-x_n),0)$.
\item $x_n$ is not a zero of $A$ and $\sigma>0$. In this case there exists $r_n>0$ such that for each $\eta_n\in X$ satisfying $\|\eta_n\|<r_n$ there exists a unique pair  $(y_n,\xi_n)\in X^2$ such that \beqref{eq:SolodovSvaiter1999} holds. In fact, 
\begin{subequations}\label{eq:y_n xi_n}
\begin{align}
y_n&=\left(I+\frac{1}{\mu_n}A\right)^{-1}\left(x_n-\frac{1}{\mu_n}\eta_n\right),\\
\xi_n&=-\eta_n-\mu_n(y_n-x_n).
\end{align}
\end{subequations}
\end{enumerate} 
 Furthermore, if $\sigma\in [0,1)$ and for some $n\in\N\cup\{0\}$ the algorithm generates $x_n$ but terminates before generating $x_{n+1}$, then $x_n$ is a zero of $A$.
\end{thm}
\begin{proof}
If $x_n$ is a zero of $A$, then $(y_n,\xi_n,\eta_n):=(x_n,0,0)$ satisfies \beqref{eq:SolodovSvaiter1999} as a simple verification shows. 

Now suppose that $\sigma=0$ (this possibility may coincide with the previous one). If some $(y_n,\xi_n,\eta_n)\in X^3$ satisfies \beqref{eq:SolodovSvaiter1999}, then \beqref{eq:eta_n<= SolodovSvaiter1999} and $\sigma=0$ imply that $\eta_n=0$. Denote $f(w):=\frac{1}{2}\|w\|^2$ for all $w\in X$.  Then $f$ is fully Legendre and $\nabla f=I$ (Examples \bref{ex:PowerRho} or \bref{ex:PositiveDefinite} above). By Proposition \bref{prop:PrincipleResolvent} (with $x=x_n$, $y_n=y$, and  $\lambda=1/\mu_n$) we conclude that  $y_n=(I+(1/\mu_n)A)^{-1}(x_n)$ and $\xi_n=-\mu_n(y_n-x_n)$. Therefore any solution to \beqref{eq:SolodovSvaiter1999} must coincide with $((I+(1/\mu_n)A)^{-1}(x_n),-\mu_n(y_n-x_n),0)$. On the other hand, Proposition \bref{prop:PrincipleResolvent} (again with $x=x_n$ and $\lambda=1/\mu_n$)  ensures that the triplet $(y_n,\xi_n,\eta_n):=((I+ (1/\mu_n)A)^{-1}(x_n),-\mu_n(y_n-x_n),0)$ does solve \beqref{eq:SolodovSvaiter1999}. 

It remains to consider the last possibility, namely,  
$0\notin Ax_n$ and $\sigma>0$. As before, set $f(w):=\frac{1}{2}\|w\|^2$ for all $w\in X$.  Denote  $\Phi(\eta,\xi,x,y):=\|\eta\|$ and $\Psi(\eta,\xi,x,y):=\sigma \max\{\|\xi\|,\mu_n\|y-x\|\}$ for all $(\eta,\xi,x,y)\in X^4$. These are continuous functions. 
Since $(I+(1/\mu_n)A)^{-1}$ is continuous (Example \bref{ex:HilbertContinuous} above), the functions $\phi:X\to\R$ and $\psi:X\to \R$ defined in \beqref{eq:y xi phi psi} with $U:=X$, $x:=x_n$, $\lambda:=1/\mu_n$ are continuous. In addition, $x_n\neq (I+(1/\mu_n)A)^{-1}(x_n)$ because the equality  $x_n=(I+(1/\mu_n)A)^{-1}(x_n)$ implies by the definition of the inverse operator that $x_n\in (I+(1/\mu_n)A)(x_n)=x_n+(1/\mu_n)A(x_n)$. Hence $0\in A(x_n)$, that is, $x_n$ is a zero of $A$, a contradiction. Therefore  $\phi(0)=0<\sigma\mu_n\|(I+(1/\mu_n)A)^{-1}(x_n)-x_n\|\leq \psi(0)$. Thus all the conditions mentioned in Proposition \bref{prop:ImplicitInexactness} are satisfied (here $\eta_n=-\eta$, $y_n=y$, $\xi_n=\xi$, $x=x_n$, $r_n=r$, $\lambda=1/\mu_n$) and there exists $r_n>0$ such that for all $\eta_n\in X$ satisfying $\|\eta_n\|<r_n$, there exists a unique pair $(y_{n},\xi_n)$ such that \beqref{eq:SolodovSvaiter1999} holds. 

Finally, suppose that $\sigma\in [0,1)$ and for some $n\in\N\cup\{0\}$ the algorithm generates $x_n$ but terminates before generating $x_{n+1}$. We know from previous lines that \beqref{eq:SolodovSvaiter1999} has solutions. Let $(y_n,\xi_n,\eta_n)\in X^3$ be such a solution. Since Algorithm \bref{alg:SolodovSvaiter1999-1} terminates, its definition implies that either $x_n=y_n$ or $\xi_n=0$. This condition and \beqref{eq:0=xi_n+} imply that $\|\eta_n\|=\|\xi_n\|$ or $\|\eta_n\|=\mu_n\|x_n-y_n\|$, and hence from \beqref{eq:eta_n<= SolodovSvaiter1999} we have $\|\eta_n\|\leq \sigma\|\eta_n\|$. Since $\sigma\in [0,1)$ it follows that $\|\eta_n\|=0$ and hence $\eta_n=0$. Therefore from \beqref{eq:0=xi_n+} we have $x_n=y_n$ and $\xi_n=0$. We conclude from \beqref{eq:xi_n in A} that $0\in A(x_n)$, as required. 
\end{proof}

\begin{remark}\label{rem:x_n is zero of A}
In the formulation of Theorem \bref{thm:SolodovSvaiter1999-1} there appears (in Case \beqref{item:x_n is zero of A}) the condition that $x_n$ is a zero of $A$. It is worthwhile  saying a few words regarding possible ways to check whether this condition holds. First, if one is able to evaluate the set $Ax_n$ and is able to check membership of elements in this set, then one  can check directly whether $0\in Ax_n$. Alternatively, one can fix an error parameter  $\epsilon_n>0$ in advance and then check whether the distance between $0$ and $Ax_n$ is less than $\epsilon_n$. If this latter condition holds, then one can regard $x_n$ as an approximate zero and terminate the algorithm.  Another way to check whether $x_n$ is a zero of $A$ is to fix $\lambda>0$ and then to consider the  equality $x_n=(I+\lambda A)^{-1}(x_n)$ which is an equality  between two elements in $X$. As can be verified directly (and was shown in the proof of Theorem \bref{thm:SolodovSvaiter1999-1}),  this equality is equivalent to the condition that $0\in Ax_n$. If $(I+\lambda A)^{-1}(x_n)$ can be evaluated, then the above-mentioned equality can be checked. If $(I+\lambda A)^{-1}(x_n)$ can be evaluated only approximately (as is common in practical scenarios), then one can fix an error parameter  $\epsilon_n>0$ in advance and then can  check whether the inequality $\|x_n-(I+\lambda A)^{-1}(x_n)\|<\epsilon_n$ holds. If this inequality holds, then $x_n$ can be regarded as an approximate zero of $A$ and we can stop the algorithm.
\end{remark}

\section{Well-definedness of Iusem, Pennanen and Svaiter \cite[Method 1, Theorem 3]{IusemPennanenSvaiter2003jour} }\label{sec:IusemPennanenSvaiter2003jour}

The paper \cite{IusemPennanenSvaiter2003jour} discusses several inexact versions of the proximal  point algorithm. The setting is a real Hilbert space $X$ and operators satisfying various monotonicity  or non-monotonicity assumptions. One of the algorithmic schemes discussed there is \cite[Method 1]{IusemPennanenSvaiter2003jour}, which is aimed at finding a zero of a maximally monotone operator $A:X\to 2^X$ assuming that $A$ has at least one zero. The scheme is defined as follows: \\

\begin{alg}\label{alg:IusemPennanenSvaiter2003jour} {\bf\noindent Initialization:} Choose an arbitrary $x_0\in X$, an arbitrary $\sigma\in [0,1)$, an arbitrary sequence of positive numbers $(\lambda_n)_{n=0}^{\infty}$ satisfying $\widehat{\lambda}:=\inf\{\lambda_n: n\in\N\}>0$, a certain positive number $\rho \in (0,\widehat{\lambda}/2)$, and define 
\begin{equation}\label{eq:nu}
\nu:=\frac{\sqrt{\sigma+(1-\sigma)\left(\displaystyle{\frac{2\rho}{\widehat{\lambda}}}\right)^2}-\displaystyle{\frac{2\rho}{\widehat{\lambda}}}}{1+\displaystyle{\frac{2\rho}{\widehat{\lambda}}}}.
\end{equation}
In addition, fix a linear subspace $Z$ in $X$.\\

{\bf \noindent Iterative step:} Given $n\in\N\cup\{0\}$ and $x_n$, find $y_n\in X$ and $\eta_n\in X$ satisfying the following conditions:
\begin{subequations}\label{eq:IusemPennanenSvaiter2003jour}
\begin{align}
\eta_n&\in (\lambda_n A(y_n)+y_n-x_n)\cap Z,\label{eq:eta in A(y_n)capZ}\\
\,\,\,\|\eta_n\|&\leq \nu \|y_n-x_n\|\label{eq:eta_n<=IPS2003}.
\end{align}
\end{subequations}
Define 
\begin{equation}\label{eq:x_n+1 IPS2003}
x_{n+1}:=y_n-\eta_n.
\end{equation}
\end{alg}

In order for the algorithm to be well defined, one should prove the existence of solutions $(y_n,\eta_n)$ to \beqref{eq:IusemPennanenSvaiter2003jour}. In \cite{IusemPennanenSvaiter2003jour} only the case of exact solutions ($\eta_n=0$) was discussed 
(in \cite[p. 1086]{IusemPennanenSvaiter2003jour}, \cite[p. 1088 (Remark 2, proof of Corollary 1)]{IusemPennanenSvaiter2003jour} and \cite[p. 1092 (proof of Corollary 3)]{IusemPennanenSvaiter2003jour}; an implicit discussion appears also in \cite[p. 1095 (above   Theorem 3)]{IusemPennanenSvaiter2003jour}). Under the assumption that there exist sequences  $(x_n)_{n=0}^{\infty}$ satisfying \beqref{eq:IusemPennanenSvaiter2003jour}--\beqref{eq:x_n+1 IPS2003} and under further assumptions, it was shown in \cite[Theorem 3(b)]{IusemPennanenSvaiter2003jour} that each such sequence $(x_n)_{n=0}^{\infty}$ converges weakly to a zero of $A$. 

We note that \cite[Method 1, pp. 1094-1095]{IusemPennanenSvaiter2003jour} is a reformulation of  \cite[Algorithm 2, pp. 1082-1083]{IusemPennanenSvaiter2003jour}. There is a slight ambiguity regarding the value of $\rho$, since in \cite[Algorithm 2]{IusemPennanenSvaiter2003jour} this value is related to a certain monotonicity assumption associated with $A$. (The issue is as follows: Both $A$ and $A^{-1}$ should be maximally  $\rho$-hypomonotone for some  $\rho\in (0,\widehat{\lambda}/2)$; this assumption is needed for the convergence analysis as can be seen in \cite[Lemma 1 and its proof (pp. 1086-1088)]{IusemPennanenSvaiter2003jour} and other results in \cite{IusemPennanenSvaiter2003jour} based on this lemma; however, in \cite[Method 1]{IusemPennanenSvaiter2003jour} $A$ is  assumed to be maximally monotone; while this implies that $A^{-1}$ is maximally monotone and thus it is also  maximally $\rho$-hypomonotone  for all arbitrary small $\rho>0$, the exact value of $\rho$ to be used in $\nu$ from \beqref{eq:nu} is  not mentioned.) Anyway, Theorem \bref{thm:IusemPennanenSvaiter2003jour} below shows that $\nu$ can be an arbitrary nonnegative number, independently of $\sigma$ and $\rho$, and usually  there is some freedom in the value of the inexact solution $(y_n,\eta_n)$. As a result, if, in  particular, we want $\nu$ to be defined by \beqref{eq:nu}, then we can take any $\sigma\in [0,\infty)$ and any $\rho \in [0,\widehat{\lambda}/2]$. It is, however, an open problem whether the sequence $(x_n)_{n=0}^{\infty}$ converges weakly to a  zero of $A$ when $\nu$ is not assumed to satisfy \beqref{eq:nu} or when it satisfies \beqref{eq:nu} but $\sigma\geq 1$, since the convergence analysis in  \cite{IusemPennanenSvaiter2003jour} depends on \beqref{eq:nu} and also on the assumption that $\sigma\in [0,1)$. 
 
\begin{thm}\label{thm:IusemPennanenSvaiter2003jour}
Consider Algorithm \bref{alg:IusemPennanenSvaiter2003jour} with any initialization,  including the case of arbitrary $\nu,\sigma\in [0,\infty)$. Then there exist sequences which satisfy this algorithm in the exact and inexact cases. More precisely, given $n\in \N\cup\{0\}$ and $x_n\in X$, at least  one of the following possibilities holds:
\begin{enumerate}[(i)]
\item $x_n$ is a zero of $A$. In this case $(y_n,\eta_n):=(x_n,0)$ satisfies \beqref{eq:IusemPennanenSvaiter2003jour};
\item $\nu=0$. In this case there exists a unique pair  $(\eta_n,y_n)\in Z\times X$ such that \beqref{eq:IusemPennanenSvaiter2003jour} holds, namely  $(y_n,\eta_n):=((I+\lambda_n A)^{-1}(x_n),0)$.
\item $x_n$ is not a zero of $A$ and $\nu>0$. In this case there exists $r_n>0$ such that for each $\eta_n\in Z$ satisfying $\|\eta_n\|<r_n$ there exists a unique $y_n\in X$ such that \beqref{eq:IusemPennanenSvaiter2003jour} holds. Furthermore, 
\begin{equation}\label{eq:y_nIPS2003}
y_n=(I+\lambda_nA)^{-1}\left(x_n+\eta_n\right).
\end{equation}
\end{enumerate} 
\end{thm}

\begin{proof}
If $x_n$ is a zero of $A$, then a simple verification shows that $(y_n,\eta_n):=(x_n,0)$ satisfies \beqref{eq:IusemPennanenSvaiter2003jour}. 

In the second possibility (which may not be disjoint from the first one) $\nu=0$. This assumption implies that if some $(y_n,\eta_n)\in X^2$ satisfies \beqref{eq:IusemPennanenSvaiter2003jour}, then  \beqref{eq:eta_n<=IPS2003} $\eta_n=0$ (in particular, $\eta_n\in Z$). Denote $f(w):=\frac{1}{2}\|w\|^2$ for all $w\in X$.  Then $f$ is fully Legendre and $\nabla f=I$ (Examples \bref{ex:PowerRho} or \bref{ex:PositiveDefinite} above). By Proposition \bref{prop:PrincipleResolvent} (with $x=x_n$ and $\eta=\eta_n/\lambda_n=0$) it follows that  $y_n=(I+\lambda_n A)^{-1}(x_n)$. Therefore any solution $(y_n,\eta_n)\in X^2$ of \beqref{eq:IusemPennanenSvaiter2003jour} must coincide with $((I+\lambda_n A)^{-1}(x_n),0)$. On the other hand, Proposition \bref{prop:PrincipleResolvent} ensures that $(y_n,\eta_n):=((I+\lambda_n A)^{-1}(x_n),0)$ does solve \beqref{eq:IusemPennanenSvaiter2003jour}. Since the pair $(y_n,\eta_n)$ exists, the right-hand side of \beqref{eq:x_n+1 IPS2003} and hence $x_{n+1}$ are well defined. 

In the third possibility $0\notin Ax_n$ and $\nu>0$. 
It must be that   $x_n\neq (I+\lambda_n A)^{-1}(x_n)$, because if 
$x_n=(I+\lambda_n A)^{-1}(x_n)$, then by the definition of the inverse operator it follows that $x_n\in (I+\lambda_n A)(x_n)=x_n+\lambda_n A(x_n)$, namely $0=\lambda_n \xi_n$ for some $\xi_n\in A(x_n)$. Because $\lambda_n\neq 0$ it follows that $\xi_n=0$ and hence $0\in A(x_n)$, that is, $x_n$ is a zero of $A$, a contradiction. Now define $f$ as above,  $\Phi(\eta,\xi,x,y):=\|\eta\|$ and $\Psi(\eta,\xi,x,y):=\nu\|y-x\|$ for all $(\eta,\xi,x,y)\in X^4$. These are continuous functions. Because $\lambda_n A$ is maximally monotone, the operator $(I+\lambda_n A)^{-1}$ is continuous (Example \bref{ex:HilbertContinuous} above). Hence the  functions $\phi:X\to\R$ and $\psi:X\to \R$ defined in \beqref{eq:y xi phi psi} with $x:=x_n$ are continuous. In addition, $\phi(0)=0<\nu\|(I+\lambda_n A)^{-1}(x_n)-x_n\|=\psi(0)$. 

Thus all the conditions mentioned in Proposition \bref{prop:ImplicitInexactness} are satisfied (with $x=x_n$ and $\lambda=\lambda_n$) and hence there exists $r>0$ such that for all $\eta\in X$ satisfying $\|\eta\|<r$, there exists a unique vector $y\in X$ such that \beqref{eq:StronglyImplicitResolventInclusion} holds. Since \beqref{eq:PerturbedResolventProblem}  is equivalent to $\lambda\eta\in \lambda A(y)+\nabla f(y)-\nabla f(x)$, if we denote $y_n:=y$, $\eta_n:=\lambda_n\eta$,  $\xi_n:=\xi$, $r_n:=\lambda_n r$ and observe that $\eta_n\in X$ satisfies $\|\eta_n\|<r_n$ if and only if $\|\eta\|<r$, we conclude from the previous discussion that for an arbitrary $\eta_n\in X$ which satisfies $\|\eta_n\|<r_n$, there exists a unique vector $y_n\in X$ such that  the relations $\eta_n\in \lambda_n A(y_n)+y_n-x_n$ and $\|\eta_n\|<\nu \|y_n-x_n\|$ are satisfied. By restricting $\eta_n$ to $Z$ we see that \beqref{eq:IusemPennanenSvaiter2003jour} holds. Since $(y_n,\eta_n)$ exists, the right-hand side of \beqref{eq:x_n+1 IPS2003} and hence $x_{n+1}$ are well defined. 
\end{proof}

\section{Well-definedness of Parente, Lotito and Solodov \cite[Algorithm 3.1, Theorems 4.2, 4.4]{ParenteLotitoSolodov2008jour}}\label{sec:ParenteLotitoSolodov2008}

The paper \cite{ParenteLotitoSolodov2008jour} discusses a variant of the proximal  point  algorithm in which the norm changes (via a positive definite matrix) at each iteration. The  setting is $X:=\R^m$, $m\in\N$, with the Euclidean norm $\|\cdot\|$ and the goal is to find a zero of a maximally monotone operator $A$, assuming that the zero set of $A$ is nonempty. The algorithmic scheme discussed there, namely, \cite[Algorithm 3.1]{ParenteLotitoSolodov2008jour}, makes uses of the notion of enlargements of set-valued operators, that is, for each $\epsilon\geq 0$, the $\epsilon$-enlargement $A^{\epsilon}$ of $A$ is defined as follows: 
\begin{equation}\label{eq:A^epsilon}
A^{\epsilon}(x):=\{y\in X: \langle y'-y,x'-x\rangle\geq -\epsilon,\,\forall x'\in X,\forall y'\in A(x')\},\quad \forall x\in X,
\end{equation}
where $\langle\cdot,\cdot\rangle$ is the standard inner product in $X$. Given a positive definite (hence symmetric) linear operator $M:X\to X$, we denote by 
 $\|\cdot\|_{M}$ the norm induced by $M$, namely,  $\|w\|_{M}:=\sqrt{\langle Mw,w\rangle}$, $w\in X$. The algorithm is defined as follows:
\begin{alg}\label{alg:PLS2008}
{\bf\noindent Initialization:} Choose arbitrary $x_0\in X$, $\sigma\in (0,1)$, $c>0$, $\theta\in (0,1)$, and two positive numbers $\lambda_{\ell}<\lambda_u$. \\

{\bf \noindent Iterative step:} Given $n\in\N\cup\{0\}$, choose a positive definite linear operator $M_n:X\to X$ satisfying $\lambda_{\ell}\leq \lambda_{\min}(M_n)\leq\lambda_{\max}(M_n)\leq \lambda_{u}$, where $\lambda_{\min}(M_n)$ and $\lambda_{\max}(M_n)$ are the minimal and maximal eigenvalues of $M_n$, respectively. Choose $c_n\geq c$ and $\sigma_n\in [0,\sigma)$. Find $(y_n,\xi_n,\eta_n)\in X^3$ and $\epsilon_n\geq 0$ satisfying the following conditions:
\begin{subequations}\label{eq:PLS2008}
\begin{align}
\xi_n&\in A^{\epsilon_n}(y_n)\label{eq:xi_n in A^epsilon(y_n)}\\
\eta_n&=c_nM_n\xi_n+y_n-x_n,\label{eq:InexactResolventInclusion PLS2008}\\
\|\eta_n\|^2_{M_n^{-1}}+2c_n\epsilon_n&\leq\sigma_n^2(\|c_nM_n\xi_n\|^2_{M_n^{-1}}+\|y_n-x_n\|^2_{M_n^{-1}}).\label{eq:eta_n<=PLS2008}
\end{align}
\end{subequations}
Now, if $y_n=x_n$, then stop. Otherwise choose $\tau_n\in [1-\theta,1+\theta]$ and define
\begin{equation}\label{eq:x_n+1 PLS2008}
a_n:=\frac{\langle\xi_n,x_n-y_n\rangle-\epsilon_n}{\|M_n\xi_n\|^2_{M_n^{-1}}},\quad x_{n+1}:=x_n-\tau_na_nM_n\xi_n.
\end{equation}
\end{alg}

In order for the algorithm to be well defined, it should be proved that there exist solutions $(y_n,\xi_n,\eta_n,\epsilon_n)$ to \beqref{eq:PLS2008} and that $\xi_n\neq 0$ whenever $y_n\neq x_n$. In \cite{ParenteLotitoSolodov2008jour} only the case of exact solutions ($\epsilon_n=0$, $\eta_n=0$) was discussed \cite[p. 243]{ParenteLotitoSolodov2008jour} by saying that the problem reduces to the exact case when $\sigma_n=0$ (and then $y_{n+1}=(I+c_nM_nA)^{-1}y_n$ as noted in \cite[p. 241]{ParenteLotitoSolodov2008jour}). Actually, if one denotes $y_n:=y$, $\xi_n:=\xi$,  $\eta_n:=0$ and $\epsilon_n:=0$, where $(y,\xi)$ is the unique solution to \beqref{eq:PrincipleSystem} (there $\eta=0$, $\lambda=c_n$ and $f(w):=\frac{1}{2}\langle M_n^{-1}w,w\rangle$ for each $w\in X$), then the quartet $(y_n,\xi_n,\eta_n,\epsilon_n)$ solves \beqref{eq:PLS2008} even if $\sigma_n>0$. However, it is not clear from  \cite{ParenteLotitoSolodov2008jour} whether there exist solutions $(y_n,\xi_n,\eta_n,\epsilon_n)$ to \beqref{eq:PLS2008} such that either  $\eta_n\neq 0$ or $\epsilon_n\neq 0$, namely solutions which are to be expected in real-world scenarios. Anyway, under the assumption that there exist  sequences $(x_n)_{n=0}^{\infty}$ satisfying \beqref{eq:PLS2008} and under additional assumptions (such as \cite[Relation (1.4)]{ParenteLotitoSolodov2008jour}; note: the parameters $\eta_k$ mentioned there are certain positive numbers which are not related to the error vectors $\eta_n$ mentioned in  \beqref{eq:PLS2008}), it was shown in \cite[Theorem 4.2]{ParenteLotitoSolodov2008jour} that $(x_n)_{n=0}^{\infty}$ converges to a zero of $A$. Under additional assumptions, a rate of convergence was established \cite[Theorem 4.4]{ParenteLotitoSolodov2008jour}.

 The following theorem shows that if no enlargements are allowed, then Algorithm \bref{alg:PLS2008} is well defined for all $\sigma\in (0,\infty]$ (including $\sigma\geq 1$), all $\theta\in\R$ (if $\theta<0$, then we interpret $[1-\theta,1+\theta]$ as the set $\{t\in\R: 1+\theta\leq t\leq 1-\theta\}$), all $c_n>0$, $n\in\N\cup\{0\}$ (not necessarily bounded away from zero by some $c>0$), and all positive definite linear operators $M_n:X\to X$, $n\in\N\cup\{0\}$ (without any restriction on their eigenvalues). (It is, however, an open problem whether the generated sequence converges to a zero of $A$ in this extended version since the convergence analysis in \cite{ParenteLotitoSolodov2008jour} depends on the assumptions imposed in Algorithm \bref{alg:PLS2008}.) In addition, if $\sigma_n>0$, then either $x_n$ is a zero of $A$ or \beqref{eq:PLS2008} has strictly inexact solutions. 
 
\begin{thm}\label{thm:PLS2008}
Suppose that $\epsilon_n=0$ for all $n\in\N\cup\{0\}$ and consider Algorithm \bref{alg:PLS2008} with any $\sigma\in (0,\infty)$ (including $\sigma\geq 1$), any $\theta\in\R$, arbitrary positive numbers $c_n$, $n\in\N\cup\{0\}$, arbitrary positive definite and symmetric linear operators  $M_n:X\to X$, $n\in\N\cup\{0\}$, and arbitrary $\sigma_n\in [0,\sigma)$. Then there exist sequences which satisfy this algorithm in the exact and inexact cases. More precisely, given $n\in \N\cup\{0\}$, at least one of the following  possibilities hold:
\begin{enumerate}[(i)]
\item  $x_n$ is a zero of $A$. In this case the triplet $(y_n,\xi_n,\eta_n):=(x_n,0,0)$ satisfies \beqref{eq:PLS2008};
\item $\sigma_n=0$. In this case there exists a unique triplet $(y_n,\xi_n,\eta_n)\in X^3$ satisfying \beqref{eq:PLS2008}, namely $((I+c_nM_n A)^{-1}(x_n),-(c_nM_n)^{-1}(y_n-x_n),0)$.\item $0\notin Ax_n$ and $\sigma_n>0$. In this case there exists $r_n>0$ such that for each $\eta_n\in X$ satisfying $\|\eta_n\|<r_n$ there exists a unique $(y_n,\xi_n)\in X^2$ such that \beqref{eq:PLS2008} holds. Moreover, 
\begin{subequations}\label{eq:y_n xi_n PLS2008}
\begin{align}
y_n&=(I+c_nM_n A)^{-1}(x_n+\eta_n),\\
\xi_n&=(c_nM_n)^{-1}\eta_n-(c_nM_n)^{-1}(y_n-x_n). 
\end{align}
\end{subequations}
\end{enumerate}
Furthermore, suppose that for some $n\in\N\cup\{0\}$ the algorithm generates $x_n$ and that $\sigma_n\in [0,1)$. If $x_n=y_n$ (namely, the algorithm terminates), then $x_n$ is a zero of $A$, and if $x_n\neq y_n$ (namely, the algorithm continues), then $\xi_n\neq 0$ (and hence $x_{n+1}$ is well defined). 
\end{thm}

\begin{proof}
If $x_n$ is a zero of $A$ and we let $(y_n,\xi_n,\eta_n):=(x_n,0,0)$, then a simple verification shows that $(y_n,\xi_n,\eta_n)$  satisfies \beqref{eq:PLS2008}. 

Suppose now that $\sigma_n=0$ (this possibility is not necessarily disjoint from the first one). If some $(y_n,\xi_n,\eta_n)\in X^3$ satisfies \beqref{eq:PLS2008}, then \beqref{eq:eta_n<=PLS2008} and $\sigma_n=0$ imply that $\eta_n=0$. Denote $f(w):=\frac{1}{2}\langle M_n^{-1}w,w\rangle$ for all $w\in X$.  Then $f$ is fully Legendre and $\nabla f=M_n^{-1}$ (Example \bref{ex:PositiveDefinite} above). The conditions of Proposition \bref{prop:PrincipleResolvent} (with $x=x_n$, $\eta=0$ and $\lambda=c_n$) are satisfied and we have $y_n=(M_n^{-1}+c_nA)^{-1}M_n^{-1}x_n$ and $\xi_n=-(c_nM_n)^{-1}(y_n-x_n)$. The expression for $y_n$ can be simplified because 
\begin{multline}\label{eq:y_n Simplified} (M_n^{-1}+c_nA)^{-1}=(M_n^{-1}\circ(I+c_nM_nA))^{-1}\\
=(I+c_nM_nA)^{-1}\circ(M_n^{-1})^{-1}=(I+c_nM_nA)^{-1}M_n.
\end{multline} 
Thus $y_n=(I+c_nM_nA)^{-1}x_n$. Therefore any  solution $(y_n,\xi_n,\eta_n)\in X^3$ of \beqref{eq:PLS2008} must coincide with the triplet  $((I+c_nM_nA)^{-1}(x_n),-(c_nM_n)^{-1}(y_n-x_n),0)$. On the other hand, Proposition \bref{prop:PrincipleResolvent} (with $x=x_n$, $\eta=0$ and $\lambda=c_n$) ensures that the above-mentioned triplet does solve \beqref{eq:PLS2008}. 

Consider the last case, that is, $0\notin Ax_n$ and $\sigma_n>0$. It must be that $x_n\neq (I+c_nM_n  A)^{-1}(x_n)$, because if $x_n=(I+c_nM_n  A)^{-1}(x_n)$, then by the definition of the inverse operator it follows that $x_n\in (I+c_nM_n A)(x_n)=x_n+c_nM_nA(x_n)$. Since $c_nM_n$ is invertible it follows that $0\in Ax_n$, a contradiction. Define $f$ as above, $\Phi:X^4\to[0,\infty)$ and $\Psi: X^4\to [0,\infty)$ by $\Phi(\eta,\xi,x,y):=\|\eta\|^2_{M_n^{-1}}$ and  $\Psi(\eta,\xi,x,y):=\sigma_n^2(\|c_nM_n\xi\|^2_{M_n^{-1}}+\|y-x\|^2_{M_n^{-1}})$ for each  $(\eta,\xi,x,y)\in X^4$. These are continuous functions. Either Example 
\bref{ex:FiniteDimension} or Example \bref{ex:HilbertContinuous} ensure that $(M_n^{-1}+c_nA)^{-1}$ is continuous. Hence the  functions $\phi:X\to\R$ and $\psi:X\to \R$ defined in \beqref{eq:y xi phi psi} (with $x=x_n$) are continuous. In addition, $\phi(0)=0<\sigma_n^2\|(I+c_nM_nA)^{-1}(x_n)-x_n\|_{M_n^{-1}}^2\leq\psi(0)$. 

Thus all the conditions mentioned in Proposition \bref{prop:ImplicitInexactness} (with $x=x_n$ and $\lambda=c_n$) are satisfied and thus there exists $r>0$ such that for each $\eta\in X$ satisfying $\|\eta\|<r$, there exists a unique pair $(y,\xi)\in X^2$ such that \beqref{eq:StronglyImplicitResolventInclusion} holds. From this $r$ we will construct in the next  paragraph $r_n>0$ such that for all $\eta_n\in X$ satisfying $\|\eta_n\|<r_n$, there is a unique pair $(y_n,\xi_n)\in X$ such that \beqref{eq:PLS2008} holds. 

 Since $c_nM_n$ is positive definite, elementary linear algebra (diagonalization) shows that so is $(c_nM_n)^2$. This fact, together with the finite dimensionality of the space,  implies that (see Example \bref{ex:HilbertContinuous}) there exists $\alpha_n>0$ such that $\langle (c_nM_n)^2x,x\rangle\geq \alpha_n\|x\|^2$ for each $x\in X$. Denote $r_n:=\sqrt{\alpha_n}r$. Fix an arbitrary $\eta_n\in X$  satisfying $\|\eta_n\|<r_n$ and let $\eta:=(c_nM_n)^{-1}\eta_n$. Then
\begin{multline}
\sqrt{\alpha_n}r=r_n>\|\eta_n\|=\|(c_nM_n)\eta\|=\sqrt{\langle (c_nM_n)\eta,(c_nM_n)\eta\rangle}\\
=\sqrt{\langle (c_nM_n)^2\eta,\eta\rangle}\geq \sqrt{\alpha_n}\|\eta\|.
\end{multline}
Hence $\|\eta\|<r$ and, as mentioned earlier, we know from Proposition \bref{prop:ImplicitInexactness} (with $\lambda=c_n$) that there exists a unique $(y,\xi)\in X$ for which \beqref{eq:StronglyImplicitResolventInclusion} holds. This pair satisfies  \beqref{eq:PrincipleExplicit}. Denote $y_n:=y$ and $\xi_n:=\xi$. Then  \beqref{eq:StronglyImplicitResolventInclusion}, the equality $\eta_n=c_nM_n\eta$ and a simple verification show that  \beqref{eq:PLS2008} holds, and we have existence. Moreover, \beqref{eq:PrincipleExplicit}, the fact that $\lambda\eta=M_n^{-1}\eta_n$ and a simplification for $y_n$ as done in \beqref{eq:y_n Simplified}, all imply \beqref{eq:y_n xi_n PLS2008}. Now, let $(y_n,\xi_n,\eta_n)$ be an arbitrary solution to \beqref{eq:PLS2008} for which $\|\eta_n\|<r_n$. We apply Proposition \bref{prop:PrincipleResolvent} which ensures that $(y_n,\xi_n)$ satisfies \beqref{eq:y_n xi_n PLS2008}, namely it coincides with the pair $(y_n,\xi_n)$ from the previous sentence. Thus whenever $\|\eta_n\|<r_n$ there exists a unique pair $(y_n,\xi_n)$ such that \beqref{eq:PLS2008} holds. 

Finally, suppose that for some $n\in\N\cup\{0\}$ the algorithm generates $x_n$ and that $\sigma_n\in [0,1)$. If $y_n=x_n$ (we already know that $y_n$ exists), then it must be that $\xi_n=0$. Indeed, from \beqref{eq:PLS2008} we have $\eta_n=c_nM_n\xi_n$ and $\|\eta_n\|^2_{M_n^{-1}}\leq\sigma_n^2\|\eta_n\|^2_{M_n^{-1}}$. Since $0\leq\sigma_n<1$, it follows that $\|\eta_n\|=0$ and thus $\eta_n=0$. Hence $\xi_n=(c_nM_n)^{-1}0=0$, as claimed. Therefore  \beqref{eq:PLS2008} implies that $0\in A(x_n)$ and thus $x_n$ is a zero of $A$. Suppose now that $y_n\neq x_n$. We already know that $(y_n,\xi_n)$ exists, but we must verify that $\xi_n\neq 0$ so that $x_{n+1}$ will be well defined. If $\xi_n=0$, then from \beqref{eq:PLS2008} we have $\eta_n=y_n-x_n$ and $\|y_n-x_n\|^2_{M_n^{-1}}\leq\sigma_n^2\|y_n-x_n\|^2_{M_n^{-1}}$.  Since $0\leq\sigma_n<1$, it follows that $\|y_n-x_n\|=0$ and therefore 
$y_n=x_n$, a contradiction. Thus indeed $\xi_n\neq 0$.
\end{proof}

\section{Well-definedness of many more algorithms}\label{sec:ManyMore}
The ideas and the results described in this paper can be applied to deduce the well-definedness of many more inexact algorithmic schemes (and corresponding convergence theorems), among them  the ones of Burachik and Iusem \cite[Algorithm IPPM: Inexact Proximal Point Method, p. 234]{BurachikIusem2008book}, Burachik, Scheimberg and Svaiter \cite[Algorithm 2.1 (Inexact Hybrid Extragradient Proximal Algorithm), Theorem 3.1]{BurachikScheimbergSvaiter2001jour},  G{\'a}rciga Otero and Iusem \cite[Inexact Proximal  Point-Extragradient Method (pp. 75--76), Theorem 3.6]{Garciga-OteroIusem2004jour}, G{\'a}rciga Otero and Svaiter \cite[Algorithm 1, Theorem 4.3]{OteroSvaiter2004jour},  Iusem and G{\'a}rciga Otero \cite[Algorithms  I, II, PI, PII,  Theorems 1-7]{IusemGarciga-Otero2002jour},\cite[Algorithms  I--IV, Theorems 1--6]{IusemOtero2001jour},  Reich and Sabach \cite[Algorithm (3.1), Theorem 3.1]{ReichSabach2009jour}, Silva, Eckstein and Humes, Jr. \cite[Box Interior Proximal Point Algorithm (BIPPA, p. 254), Theorem 4.10 (p.  255)]{daSilvaSilvaEcksteinHumes2001jour}, and Solodov and Svaiter \cite[Algorithm 3.1 (Hybrid Approximate  Extragradient-Proximal Point Algorithm, pp. 331-332), Theorems  3.1--3.2]{SolodovSvaiter1999-2jour}, \cite[Relation (9), Theorem 1]{SolodovSvaiter2000incol},\cite[Algorithm 2.1, Theorems 6, 8]{SolodovSvaiter2001jour}. Our results can be applied also in the context of Griva and Polyak \cite[The modified PPNR Method (8)--(10) (p.  285), Theorem 4.10]{GrivaPolyak2011jour} and  Rockafellar \cite[Algorithm (B)  (p. 880), Theorems 2-3]{Rockafellar1976jour}, but they seem not very natural in the context of these latter papers.

We believe, but leave it as an open problem for a future investigation,  that modifications and  generalizations of the methods presented here may be applied in one way or another in the context of many other inexact algorithmic schemes (and corresponding convergence results) which can be found in the literature. These schemes  are closely related, but somewhat different from  the ones presented in this paper (due to different imposed assumptions, say those related to the relevant  operators or sequences), and in the majority of them there are issues with their well-definedness when non-zero error terms appear. Among the schemes which seem promising in this  context are the ones presented in Auslender, Teboulle, and Ben-Tiba  \cite[The Logarithmic-Quadratic Proximal  method (LQP, p. 34), Theorem 1]{AuslenderTeboulleBen-Tiba1999jour}, Burachik, Lopes, and Da Silva \cite[Extragradient Algorithm (EA, p. 26), Theorem 3.11]{BurachikLopesDa-Silva2009jour},  Burachik and Svaiter \cite[Hybrid Interior Proximal 
Extragradient Method (HIPEM), p. 820, Theorem 4.1]{BurachikSvaiter2001jour},  Eckstein and Svaiter \cite[Algorithm 3, Proposition 4.2]{EcksteinSvaiter2009jour}, G{\'a}rciga Otero and Iusem \cite[Algorithms 1--2, Theorems 1--3]{GarcigaOtero-Iusem2007jour}, \cite[Algorithms 1,2  (Section 4), Theorems 4.1--4.2]{OteroIusem2013jour}, Humes, Silva and Svaiter \cite[The hybrid algorithms of Subsection 2.2, Theorems 1-2]{HumesSilvaSvaiter2004jour}, Lotito, Parente, and Solodov 
\cite[The algorithm on p. 860, Algorithm 2.1 (VMHPDM, pp. 862--863), 
Theorem 2.2]{LotitoParenteSolodov2009jour}, Monteiro and Svaiter \cite[Large-step HPE Method (pp. 917--918), Inexact NPE Method (p. 922), Theorems 2.5,2.7,3.5, 3.6]{MonteiroSvaiter2012jour}, \cite[A-HPE framework (pp. 1095-1096), Large-step A-HPE framework (p. 1102), Theorems 3.6,3.8,4.1]{MonteiroSvaiter2013jour}, Solodov \cite[Hybrid proximal decomposition method (HPDM, Algorithm 2.1, pp. 561--562)]{Solodov2004jour}, Solodov and Svaiter \cite[Algorithm 1 (p. 384), Theorem 3]{SolodovSvaiter2000b-jour}, \cite[Algorithm 1 (Inexact Generalized Proximal Method, p. 222), Theorem 3.2]{SolodovSvaiter2000jour}, and Xia and Huang \cite[Algorithm 3.1 (pp.  4598--4599), Theorem 4.5]{XiaHuang2011jour}.

\section{Concluding remarks}\label{sec:AdditionalRemarks}
We conclude the paper with the following remarks.

\begin{remark}
It will be interesting and useful to extend the ideas and various assertions described  in this paper to  other settings. In particular, to allow (with a suitable caution due to the presence of error terms) in the inexact resolvent problem \beqref{eq:PerturbedResolventProblem} functions $f$ having effective domains which are subsets of the whole space, to allow  enlargements of operators (here it seems reasonable to extend the theory of resolvents mentioned briefly in Section \bref{sec:Preliminaries} and the references cited there to resolvents of enlargements, and \cite{BurachikIusemSvaiter1997jour,BurachikSagastizabalSvaiter1999jour,BurachikIusem2008book} may be of some help in this direction), to consider also inexactness coming from  $\epsilon$-subdifferentials, to allow spaces more general than normed spaces such as Hadamard spaces and other metric spaces  \cite{AhmadiKhatibzadeh2014jour,Bacak2013jour,Bacak2014book,LiLopezMartin-Marquez2009jour,TangHuang2014jour,WangLiLopezYao2015jour,Zaslavski2011-2jour} (the theory of resolvents for Hadamard  spaces described in \cite{LiLopezMartina-MarquezWang2011jour} may help in this direction), to allow certain nonlinear modifications of \beqref{eq:PerturbedResolventProblem} such as the one given in \cite[p. 179]{Aragon-ArtachoGeoffroy2007jour} and \cite[pp. 648, 650, 658]{AuslenderTeboulleBen-Tiba1999-2jour} (and to extend the latter ones so they will allow   general Bregman distances which may not be induced from Bregman functions  \cite{Reem2012incol}), to allow inducing functions $f$ more general than fully Legendre such as zero-convex functions \cite{CensorReem2015jour} (or at least special but important classes of zero-convex functions such as quasiconvex functions  \cite{Papa-QuirozOliveira2009jour}), d.c. functions \cite{SouzaOliveira2015jour}, and so on.
\end{remark}

\begin{remark}\label{rem:ReichSabachFullyLegendre} 
In the case of \cite[Theorem 4.1]{ReichSabach2010b-jour} the assumptions on $f$ mentioned in Section \bref{sec:ReichSabach2010b4.1} above imply that $f$ is fully Legendre. Indeed, $f$ is assumed to be totally convex  and therefore it is convex ($f^*$ is always convex); the interior of the effective domains of $f$ and $f^*$ are $X$ and $X^*$, respectively, and therefore both functions are proper; in addition, $\nabla f$ and $\nabla f^*$ are defined on $X$ and $X^*$, respectively, because $f$ and $f^*$ are assumed to be Fr\'echet differentiable and hence both functions are  G\^ateaux differentiable;  as a result, $\dom(\nabla f)=X$ and $\dom(\nabla f^*)=X^*$; since both $f$ and $f^*$ are Fr\'echet differentiable and thus also lower  semicontinuous, we conclude from the above discussion that $f$ is fully Legendre.
\end{remark}

\begin{remark}\label{rem:AuslenderTeboulleBen-Tiba}
In Remark \bref{rem:HistoryInexactResolventInclusionProblem} above we mentioned Auslender et al. \cite{AuslenderTeboulleBen-Tiba1999jour} and briefly discussed its relation to the inexact resolvent inclusion problem \beqref{eq:PerturbedResolventProblem}. Here we want to discuss additional issues related to \cite{AuslenderTeboulleBen-Tiba1999jour} and to our paper. First, the setting in  \cite[Proposition 2]{AuslenderTeboulleBen-Tiba1999jour} (see also Auslender and Teboulle \cite[Prop. 6.8.3, pp. 216--217]{AuslenderTeboulle2003book}) is a finite-dimensional Euclidean space $X$, a  maximally monotone operator $A$ the effective  domain of which intersects the effective domain of a certain linear deformation of $f$, the  function $f$ is a proper lower semicontinuous convex function which is (Fr\'echet) differentiable on its nonempty and open effective domain, its gradient is onto $X$, and its recession function $f_{\infty}$ satisfies $f_{\infty}(x)=\infty$ for all $x\neq 0$. According to Remark \bref{rem:FullyLegendreFIniteDim} above, if we also assume that the effective domain of $f$ is the whole space and $f$ is strictly convex  there, then $f$ must be fully Legendre. Second, although \cite[Proposition 2]{AuslenderTeboulleBen-Tiba1999jour} allows the effective domain of $f$ to be a strict subset of the space, in this case caution is needed before one can apply \cite[Proposition 2]{AuslenderTeboulleBen-Tiba1999jour} to the inexact resolvent inclusion problem or to some iterative algorithms, because the error terms may induce points located outside the effective domains of certain key operators. 
\end{remark}

\begin{remark}
In addition to \cite[Algorithm (4.1), Theorem 4.1]{ReichSabach2010b-jour}, the paper \cite{ReichSabach2010b-jour} contains another algorithmic scheme and a  corresponding strong convergence theorem, namely \cite[Algorithm (4.4), Theorem 4.2]{ReichSabach2010b-jour}. Although it is not entirely clear from the formulations of the scheme and the theorem that the error terms  mentioned there can be arbitrary, a simple verification shows that they indeed can. Moreover, there is no need to make any modification in the corresponding formulations and proof (and, in particular,  there is no need to use  any external result such as Proposition \bref{prop:PrincipleResolvent} above). Similar observations hold regarding the various algorithmic schemes and strong convergence results established in \cite{ReichSabach2010jour,ReichSabach2012col}. 
\end{remark}

\begin{remark}
It would be of interest to develop further the continuity results discussed in Section \bref{sec:Continuity}. For instance, to give additional sufficient conditions which guarantee the continuity of the protoreolvent, to find examples where it is discontinuous (or to prove that such examples are impossible), and to establish results in which not only the vectors $x$  and $\eta$ are allowed to vary, but also the relaxation parameter $\lambda$, the operator $A$ and the function $f$. 
\end{remark}

\section*{Acknowledgments}
We would like to express our thanks to Shoham Sabach and Roman Polyak for helpful discussions, and to the referees for considering our paper and for their feedback. Part of the research of the first author was done while he was at the Institute of Mathematical and Computer Sciences (ICMC), University of S\~ao Paulo,  
 S\~ao Carlos, Brazil (2015) and this is an opportunity for him to thank FAPESP.  The second author was partially supported by the Israel Science Foundation (Grant 389/12), by the Fund for the Promotion of Research at the Technion and by the Technion General Research Fund.



\end{document}